\documentclass[11pt]{amsart}

\usepackage{amssymb,latexsym}
\usepackage{amsthm}
\usepackage{amsmath}
\usepackage{amsfonts}
\usepackage{graphicx}
\usepackage[all]{xy}
\usepackage{chngcntr}
\usepackage{fancyhdr}
\usepackage[top=1in,bottom=1in,left=1.25in,right=1.25in]{geometry}

\newtheorem{theorem}{Theorem}[subsection]
\newtheorem{definition}[theorem]{Definition}
\newtheorem{remark}[theorem]{Remark}
\newtheorem{proposition}[theorem]{Proposition}
\newtheorem{corollary}[theorem]{Corollary}
\newtheorem{lemma}[theorem]{Lemma}
\newtheorem{example}[theorem]{Example}
\newtheorem{examples}[theorem]{Examples}
\newtheorem{claim}[theorem]{Claim}

\def\Q{\mathbb{Q}}

\def\R{\mathbb{R}}
\def\Z{\mathbb{Z}}
\def\A{\mathbb{A}}

\def\C{\mathbb{C}}
\def\G{\mathbb{G}}

\def\Al{\mathcal{A}}

\def\Fl{\mathcal{FL}}

\def\S{\mathcal{S}}

\def\ch{\mathrm{ch}}
\def\CAff{\mathrm{CAff^{op}}}
\def\Lie{\mathrm{Lie}}
\def\M{\mathrm{M}}
\def\Perf{\mathrm{Perf}}
\def\Sh{\mathrm{Sh}}
\def\GSp{\mathrm{GSp}}
\def\SO{\mathrm{SO}}
\def\GSpin{\mathrm{GSpin}}
\def\Spin{\mathrm{Spin}}
\def\Spa{\mathrm{Spa}}

\def\Spec{\mathrm{Spec}}
\def\Adic{\mathrm{Adic}}

\def\lan{\langle}
\def\ran{\rangle}

\def\lra{\longrightarrow}
\def\ra{\rightarrow}
\def\ov{\overline}

\def\wh{\widehat}
\def\wt{\widetilde}
\def\st{\stackrel}
\def\tr{\textrm}

\usepackage[pdftex]{hyperref}

\begin{document}

\title{Perfectoid Shimura varieties of abelian type}
\author{Xu Shen}
\date{}
\address{Fakult\"at f\"{u}r Mathematik\\
Universit\"at Regensburg\\
Universitaetsstr. 31\\
93040, Regensburg, Germany} \email{xu.shen@mathematik.uni-regensburg.de}

\address{Current address: Morningside Center of Mathematics\\
	No. 55, Zhongguancun East Road\\
	Beijing 100190, China}\email{shen@math.ac.cn}
	
\renewcommand\thefootnote{}

\renewcommand{\thefootnote}{\arabic{footnote}}
\begin{abstract}
We prove that Shimura varieties of abelian type with infinite level at $p$ are perfectoid. As a corollary, the moduli spaces of polarized K3 surfaces with infinite level at $p$ are also perfectoid.
\end{abstract}

\maketitle
\tableofcontents

\section{Introduction}
\subsection{Background and motivation}
The theory of perfectoid spaces was originally developed by Scholze in \cite{Sch1} to prove some cases of the weight-monodromy conjecture over $p$-adic fields. Since then, this theory has proved to be very useful in quite a lot areas of number theory and arithmetic geometry, cf. \cite{Sch2, Sch3, SW}. In \cite{Sch3}, Scholze has proved that Shimura varieties of Hodge type with infinite level at $p$ are perfectoid, which is the key geometric ingredient of his construction of automorphic Galois representations there. More precisely, in \cite{Sch3} Scholze constructed Galois represenations associated with the mod $p^m$ cohomology of the locally symmetric spaces for $\mathrm{GL}_n$ over a totally real or CM field, proving conjectures of Ash and others on the mod $p$ version of the global Langlands conjecture. The cohomology of the locally symmetric spaces for $\mathrm{GL}_n$ can be realized as a boundary contribution of the cohomology of symplectic or unitary Shimura varieties (which are of Hodge type), and the perfectoid structure on these Shimura varieties with infinite level at $p$ plays a crucial role in understanding the torsion appearing in the cohomology. In this note, we prove that Shimura varieties of abelian type with infinite level at $p$ are also perfectoid. As a corollary, we get many new interesting examples of perfectoid spaces, including perfectoid quaternionic Shimura varieties and the moduli spaces of polarized K3 surfaces with infinite level at $p$.

Shimura varieties of abelian type are exactly those studied by Deligne in \cite{D}, where he proved that the canonical models of these Shimura varieties exist. When the weight is rational, Shimura varieties of abelian type (over characteristic 0) are known as moduli spaces of abelian motives (defined by using absolute Hodge cycles), cf. \cite{Mi2} Theorem 3.31. Recently, Kisin (cf. \cite{Ki}) and Vasiu (cf. \cite{V1, V2}, as well as the more recent \cite{V3,V4}) have proved that integral canonical models for Shimura varieties of abelian type (in the case of hyperspecial levels at $p$) exist. These works and the results of Scholze in \cite{Sch3} are the motivation of this paper. Recall that Shimura varieties of Hodge type are the Shimura varieties which can be realized as closed subspaces of the Siegel Shimura varieties. In principle, Deligne's paper \cite{D} suggests that, if a (reasonable) statement is true for Shimura varieties of Hodge type, then it should also be true for Shimura varieties of abelian type, cf. \cite{MS, Ki, Mo} for examples. However, the situations in \cite{Ki} and \cite{Sch3} are quite different. Let $(G, X)$ be a Shimura datum. Fix a prime $p$. In \cite{Ki}, one fixes the level $K_p$ at $p$ as the hyperspecial level (thus the reductive group is unramified at $p$), and lets the level $K^p$ outside $p$ vary. Equivalently, one studies the pro-variety
\[\Sh_{K_p}(G,X)=\varprojlim_{K^p}\Sh_{K_pK^p}(G,X)\]over the reflex field and the integer ring of the local reflex field. Here, $K^p\subset G(\A_f^p)$ runs through open compact subgroups of $G(\A_f^p)$. In the situation of \cite{Sch3}, one fixes the level $K^p$ outside $p$ and lets the level $K_p$ at $p$ vary. Equivalently, one studies the object \footnote[1]{Since the usual projective limit does not exist in the category of adic spaces: the closest object to this limit is as in Definition \ref{def}. See also Lemma \ref{L:existence}.}
\[\varprojlim_{K_p}\Sh_{K_pK^p}(G,X)^{ad}\] in the pro-\'etale site of the adic Shimura varieties $\Sh_{K_pK^p}(G,X)^{ad}$  over a perfectoid field like $\C_p$. Here, $K_p\subset G(\Q_p)$ runs through open compact subgroups of $G(\Q_p)$. In this paper, we will mainly work with the latter situation. Nevertheless, we will use in the course of the argument (cf. Proposition \ref{P:delta} in subsection 3.2) the case where a hyperspecial level $K_\ell$ is fixed, for a prime $\ell\neq p$, and the level $K^\ell$ varies.

Recall that $(G, X)$ is called of abelian type, if there exists a Shimura datum $(G_1,X_1)$ of Hodge type, together with a central isogeny $G_1^{der}\ra G^{der}$ which induces an isomorphism between the associated adjoint Shimura data $(G_1^{ad},X^{ad}_1)\st{\sim}{\ra}(G^{ad},X^{ad})$. Thus the geometry of Shimura varieties of abelian type and of Hodge type are very closely related. In fact, in \cite{D, Ki, Mo} Deligne, Kisin and Moonen all studied connected Shimura varieties of Hodge type first. Then passing to a compact (finite) quotient, they got connected Shimura varieties of abelian type. By the theory of connected components of Shimura varieties, they got results for the non-connected Shimura varieties of abelian type. In this paper, we adapt their methods and construction to the situation studied in \cite{Sch3}. 

\subsection{Main results and the strategy}
Let $(G, X)$ be a Shimura datum of abelian type. Fix a sufficiently small prime to $p$ level $K^p\subset G(\A_f^p)$. As an orientation,
we first study the connected Shimura variety \[\Sh_{K^p}^0(G,X)=\varprojlim_{K_p}\Sh_{K_pK^p}^0(G,X)\] as a pro-variety over $\C_p$, following the line as in \cite{D, Ki, Mo}.
Let $(G_1, X_1)$ be the Shimura datum of Hodge type as above. Then we can find an open compact subgroup $K_1^p\subset G_1(\A_f^p)$, such that
\[\Sh_{K^p}^0(G,X)=\Sh_{K^p_1}^0(G_1,X_1)/\Delta\] for some \emph{finite} group (cf. Proposition \ref{P:delta} in subsection 3.2) $\Delta$ which acts freely on the connected component $\Sh_{K^p_1}^0(G_1,X_1)$ of $\Sh_{K^p_1}(G_1,X_1)=\varprojlim_{K_{1p}}\Sh_{K_{1p}K^p_1}(G_1,X_1)$. The finiteness of $\Delta$ will be crucial for our later use. Now we want a perfectoid version of this construction. First of all, with some effort, we can show that there is a perfectoid Shimura variety $S_{K^p_1}^0(G_1,X_1)$, which occurs as in some sense a ``connected component'' of the perfectoid Shimura variety $S_{K^p_1}(G_1,X_1)$ constructed by Scholze in \cite{Sch3}. Moreover, it is equipped with a free action of $\Delta$, induced from the algebraic situation. In fact, what we are doing is to adapt the construction of \cite{D} 2.7 to the perfectoid situation. This (natural and not very hard) adaption will also be a key tool in our study of perfectoid Shimura varieties here. Then, we want to define an adic space by taking the quotient on the connected perfectoid Shimura variety $S_{K^p_1}^0(G_1,X_1)$ by $\Delta$:
\[S_{K^p}^0(G, X)=S_{K^p_1}^0(G_1, X_1)/\Delta,\]
which is the adic version of $\Sh_{K^p}^0(G,X)$. 

At this point, some curious reader may wonder that, why do not we first perform the (algebraic) quotients at finite levels, and then take the limit of the associated adic spaces to get a perfectoid space as what we want? If so, in fact one will not need to take quotients at all, since the desired objects at finite levels are given by the connected Shimura varieties $\Sh_{K_pK^p}^0(G,X)$, and it is by the theory of Shimura varieties that we know each $\Sh_{K_pK^p}^0(G,X)$ can be realized as a finite quotient of some suitable $\Sh_{K_{1p}K^p_1}^0(G_1,X_1)$, for details see subsections 3.1 and 3.2. Therefore, if one wants to take this approach, then the only problem is prove the existence of the limit $\varprojlim_{K_p}\Sh_{K_pK^p}^0(G,X)^{ad}$ in the sense of Definition \ref{def}, as a perfectoid space. This is exactly what we want, up to our adaption of Deligne's construction above! Then, one meets an essential difficulty: it is not clear at all that the limit (in the sense of Definition \ref{def}) exists, even as an adic space! See Lemma \ref{L:existence} (and Lemma \ref{L:perfectoid-local}) for some discussion on the existence of limits of adic spaces in some cases.  In fact, even in Lemma \ref{L:existence}, one is asserting something (the existence of a limit in the sense of Definition \ref{def}) a bit weaker than the existence of the usual limit, but it is the unique possibility for what an adic-space limit could be (see also the footnote 1 in the second paragraph of this introduction on comparison with Kisin's work \cite{Ki}). Here one finds some subtleties in the world of adic spaces.  Thus, we come back to our previous approach: try to construct the quotient at infinite level $S_{K^p_1}^0(G_1, X_1)/\Delta$ first as an adic space, and then prove that it is perfectoid. As a summary, with all of our efforts so far, we have reduced our specific problem in the setting of Shimura varieties to a problem on quotients of perfectoid spaces.

However, it is not obvious at all that such a quotient above $S_{K^p_1}^0(G_1, X_1)/\Delta$ as adic space exists! This leads us to study group quotients of perfectoid spaces seriously. Here come the essential new difficulties relative to known methods in the algebraic setting. First, we remark that in the classical rigid analytic geometry setting and assuming that the spaces are separated, by resorting to the associated Berkovich spaces, there exist nice results on \'etale quotients of these spaces (and their counterparts in other versions of $p$-adic geometry, e.g. Berkovich spaces and adic spaces) in \cite{CT}. However, we can not apply these affirmative results in \cite{CT}, as we are now working with perfectoid spaces, which are adic spaces almost never of finite type (over a fixed perfectoid field $k$), while the adic spaces associated to rigid analytic spaces and Berkovich spaces (as studied in \cite{CT}) are of finite type over the base field $k$.  By using the notion of diamonds introduced in \cite{Sch4}, the group quotient above exists naturally as a diamond. However, since we want to find some new perfectoid spaces coming from Shimura varieties, we will not make use of the theory of diamonds in this paper.  In our specific situation, thanks to the fact that $\Delta$ is \emph{finite}, we can show that there indeed exists such an adic space\[S_{K^p}^0(G, X),\] which may be viewed as a quotient space of $S_{K^p_1}^0(G_1, X_1)$ by $\Delta$, see Proposition \ref{P:quotient}, Corollary \ref{C:Quotient} and Proposition \ref{P:abelian-connected}. In fact, in Proposition \ref{P:quotient} (and Corollary \ref{C:quotient}) we prove a general result on the existence of \'etale quotients for finite free group actions on perfectoid spaces (over a perfectoid field). By construction, we have a finite \'etale Galois cover \[S_{K^p_1}^0(G_1, X_1)\ra S_{K^p}^0(G, X)\] with Galois group $\Delta$. In such a situation, as $S_{K^p_1}^0(G_1, X_1)$ is a perfectoid space,
by a theorem of Kedlaya-Liu (\cite{KL} Proposition 3.6.22, which is the affinoid situation here), $S_{K^p}^0(G, X)$  is also a perfectoid space. Then it will be not hard to construct a perfectoid Shimura variety $S_{K^p}(G,X)$ from $S_{K^p}^0(G,X)$, by using the theory of connected components of Shimura varieties (see the previous paragraph on our adaption of Delgine's construction). Moreover, there is naturally a Hodge-Tate period map
\[\pi_{\tr{HT}}: S_{K^p}(G,X)\lra \Fl_G,\]where $\Fl_G$ is the $p$-adic flag variety associated to the Shimura datum $(G, X)$, see subsection 3.4.
The main theorem of this paper is as follows, cf. Theorems \ref{S:thm}, \ref{HT}.
\begin{theorem} Assume that $(G,X)$ is an abelian type Shimura datum.
\begin{enumerate}
\item There is a perfectoid space $S_{K^p}=S_{K^p}(G,X)$ over $\C_p$ such that
\[S_{K^p}\sim \varprojlim_{K_p}\Sh_{K_pK^p}(G,X)^{ad},\]where $\Sh_{K_pK^p}(G,X)^{ad}$ is the adic space associated to $\Sh_{K_pK^p}(G,X)$ over $\C_p$, and the meaning of $\sim$ is as the Definition 2.4.1 of \cite{SW}. 
\item There is a $G(\Q_p)$-equivariant map of adic spaces
\[\pi_{\emph{HT}}: S_{K^p}\lra \Fl_G,\]compatible with the construction in \cite{CS} in the case that $(G, X)$ is of Hodge type. The map $\pi_{\emph{HT}}$ is invariant for the prime to $p$ Hecke action on $S_{K^p}$. Moreover,  pullbacks of automorphic vector bundles over finite level Shimura varieties to $S_{K^p}$ can be understood by using the map $\pi_{\emph{HT}}$.
\end{enumerate}
\end{theorem}

In fact, we can say more about the theorem above in the general setting.
Let $(G,X)$ be a Shimura datum. Fix a prime to $p$ level $K^p\subset G(\A_f^p)$. Consider the statement
\[\mathcal{P}(G,X): \tr{There exists a perfectoid space}\, S_{K^p}\, \tr{over}\, \C_p \,\tr{such that}\]\[ 
	S_{K^p}\sim \varprojlim_{K_p}\Sh_{K_pK^p}(G,X)^{ad}.\]
Fix a connected component $X^+\subset X$, and consider the statement
\[\mathcal{P}(G^{der},X^+): \tr{There exists a perfectoid space}\, S_{K^p}^0\, \tr{over}\, \C_p \,\tr{such that}\]\[ S_{K^p}^0\sim \varprojlim_{K_p}\Sh_{K_pK^p}^0(G,X)^{ad}.\]
Then we can prove that the two statements are equivalent
	\[\mathcal{P}(G,X)\Longleftrightarrow \mathcal{P}(G^{der},X^+).\]
And to prove the statement $\mathcal{P}(G,X)$, it suffices to work in the case that $G$ is simply connected, cf. Corollaries \ref{C:P} and \ref{C:shim}.
This opens the way to prove a theorem like the above toward all Shimura varieties.
Here, we assume that $(G, X)$ is of abelian type in order to use the theorem of Scholze on Hodge type Shimura varieties as the input.

Let $d$ be a positive integer, and $L_d$ be the quadratic lattice over $\Z$ of discriminant $2d$ and rank 21 introduced in 2.10 of \cite{MP}, see also subsection 4.2. Let $G=\SO(L_d\otimes\Q)$. Then there exists a Hermitian symmetric domain $X$ such that $(G, X)$ forms a Shimura datum of abelian type with reflex field $\Q$. Fix a tame (i.e. sufficiently small) level $K^p\subset G(\A_f^p)$. For any open compact subgroup $K_p\subset G(\Q_p)$ such that $K_pK^p$ is admissible (see subsection 4.2), let $\M_{2d,K_pK^p}$ be the moduli space of K3 surfaces together with a polarization of degree $2d$ over $\Q$, and let $\Sh_{K_pK^p}$ be the Shimura variety of level $K_pK^p$ associated to the datum $(G, X)$ over $\Q$. By Corollaries 4.4 and 4.15 of \cite{MP}, there is a period map over $\Q$
\[\M_{2d,K_pK^p}\lra \Sh_{K_pK^p}.\]More importantly, this period map is an open immersion. This is essentially the global Torelli theorem for K3 surfaces. Let $\M_{2d, K_pK^p}^{ad}$ be the adic space associated to $\M_{2d,K_pK^p}\times\C_p$. As a corollary of the above theorem, we get the following interesting perfectoid space.
\begin{corollary}
 There is a perfectoid space $\M_{2d, K^p}$ over $\C_p$ such that
\[\M_{2d,K^p}\sim \varprojlim_{K_p}\M_{2d, K_pK^p}^{ad}.\]
\end{corollary}
By construction, we have an open immersion of perfectoid spaces over $\C_p$
\[\M_{2d,K^p}\subset S_{K^p}, \]where $S_{K^p}$ is the perfectoid Shimura variety with prime to $p$ level $K^p$ associated to the datum $(G, X)$. In particular, the restriction on $\M_{2d,K^p}$ of the Hodge-Tate period map $\pi_{\tr{HT}}$ for $S_{K^p}$ gives rise to a Hodge-Tate period map
\[\pi_{\tr{HT}}: \M_{2d,K^p}\ra \Fl_G,\]which can be understood by the Kuga-Satake construction for K3 surfaces, and the Hodge-Tate period map of \cite{CS} (and \cite{Sch3}).

The same method can be used to prove that some other moduli spaces of polarized higher dimensional Calabi-Yau varieties with infinite level at $p$ are perfectoid: use perfectoid Shimura varieties of abelian type and the global Torelli theorem for the period map as the input. See the last paragraph of subsection 4.3.

In \cite{Sch3}, Scholze has also proved stronger versions of the above theorem for some compactification of Hodge type Shimura varieties (but less information for the Hodge-Tate period map). In the Siegel case, it is the minimal compactification studied there. For general Hodge type Shimura varieties, it is not known that whether the compactification used in \cite{Sch3} is the minimal one. Here, we feel that the issue of compactification will require an independent treatment, and we will leave it to a future work.

\subsection{Overview of the paper}
We now briefly describe the content of this article. In section 2, we make some preparation on adic and perfectoid spaces. In particular, we prove the existence of the \'etale quotient of a perfectoid space under a free action of a finite group. In section 3, we first review the constructions of Deligne, Moonen and Kisin as the motivation of our construction. Then after the preparation of the construction on the scheme level, we prove our main theorems on perfectoid Shimura varieties of abelian type and the Hodge-Tate period map on them. In section 4, we study an example, namely Shimura varieties of orthogonal type. In a special case, these varieties are closely related to the moduli spaces of polarized K3 surfaces. Then we prove that these moduli spaces with infinite level at $p$ are also perfectoid, by using our main theorem.
 
\subsection{Acknowledgments} I would like to thank Laurent Fargues for his constant encouragement and support in mathematics. I would also like to thank Peter Scholze sincerely, for his encouragement, suggestions, and his series works with deep insight on perfectoid spaces, which provide the cornerstone for the birth of this article. I should thank Ruochuan Liu for pointing out the precise reference in \cite{KL}. I wish to express my gratitude to Beno\^{\i}t Stroh for pointing out some gaps of arguments on adic spaces in the previous version. Also, I want to thank Liang Xiao for proposing some useful questions. I would like to thank Brian Conrad for helpful comments which make some parts of this paper more clear. Finally, I should thank the referee(s) for careful reading and valuable suggestions. This work was supported by the SFB 1085 ``Higher Invariants'' of the DFG.
\

\section{Some preliminaries on adic spaces}
We make some preparation on adic spaces in this section. 
We will use the conventions of Scholze on adic spaces in \cite{Sch4} section 4, see also \cite{SW} 2.1. 

\subsection{A brief review of adic spaces}
Consider the category $\CAff$, opposite of the category of complete Huber pairs $(A,A^+)$. We give it the structure of a site by declaring a cover of $(A,A^+)$ to be a family of morphisms $(A,A^+)\ra (A_i,A_i^+)$, such that $(A_i,A_i^+)=(\mathcal{O}_X(U_i),\mathcal{O}_X^+(U_i))$ for a covering of $X=\Spa_{top}(A,A^+)$ by rational subsets $U_i\subset X$. Here $\Spa_{top}(A,A^+)$ is the topological space (adic spectrum) associated to $(A,A^+)$, i.e. the set of equivalence classes of continuous valuations $|-|$ on $A$ which satisfy $|A^+|\leq 1$, equipped the topology generated by rational subsets. Let $\Spa(A,A^+)$ be the sheafification of the presheaf
\[(B,B^+)\mapsto \mathrm{Hom}((A,A^+),(B,B^+)) \]
on the site $\CAff$. An adic space is a sheaf on $\CAff$ such that \emph{locally} (cf. \cite{SW} Definition 2.1.5 for the precise meaning) it is of the form 
\[\Spa(A,A^+).\]  
The category of adic spaces is denoted by $\Adic$.
In particular, this category is larger than that studied by Huber in \cite{H}, where Huber restricted to the adic spaces built from those $(A,A^+)\in \CAff$ such that the pre-structure sheaves $\mathcal{O}_{\Spa_{top}(A,A^+)}$ are sheaves. The adic spaces in loc. cit. are called honest adic spaces in \cite{Sch4} and \cite{SW}. The category of honest adic spaces is denoted by $\Adic^h$, which is a full subcategory of $\Adic$.

Let $k$ be a non-archimedean field, by which we mean a field that is complete with respect to a nontrivial non-archimedean rank one valuation. We will work with the category \[\Adic/\Spa(k,k^\circ)\] of adic spaces over $k$, where $k^\circ$ is the integer ring of $k$. We denote this category as $\Adic_k$ for simplicity. Correspondingly, we consider complete Huber pairs over $(k, k^\circ)$, and denote this category as $k-\CAff$. We have the full subcategory $\Adic_k^h$ of honest adic spaces over $k$. The full subcategory of $\Adic_k$ (resp. $\Adic_k^h$) consisting of analytic adic (resp. analytic honest adic) spaces over $k$ will be denoted as $\Adic_k^a$ (resp. $\Adic_k^{ah}$). By \cite{Sch4} Proposition 4.5.2,  any analytic adic space $X$ over $k$ is the colimit of a diagram of affinoid adic spaces $\Spa(A,A^+)$, with $A$ Tate $k$-algebras. We remark that the adic spaces studied in \cite{Sch1} are exactly objects in $\Adic_k^{ah}$.

Recall that an adic space $X$ over $k$ is called locally noetherian, if it is locally of the form $\Spa(A,A^+)$, where $A$ is a strongly noetherian Tate $k$-algebra. These spaces form a full subcategory of $\Adic_k^{ah}$, which contains the essential image of the fully faithful functor \[\{\tr{rigid analytic spaces over}\, k\}\ra \Adic_k^{ah}\] associating a rigid analytic space its corresponding adic space over $k$, cf. \cite{H}, p.42.

When $k$ is a perfectoid field, there is another full subcategory $\Perf_k$ of $\Adic_k^{ah}$, the category of perfectoid spaces over $k$, cf. \cite{Sch1} Theorem 6.3.

\subsection{Inverse limits of adic spaces}
We recall the following definition, which will be crucial to our study of perfectoid Shimura varieties in the next section. 
 \begin{definition}[\cite{SW} Definition 2.4.1]\label{def}
 	Let $(X_i)_{i\in I}$ be a filtered inverse system of adic spaces with quasicompact and quasiseparated transition maps, let $X$ be an adic space, and let $f_i: X\ra X_i$ be a compatible family of morphisms. We write $X\sim \varprojlim_i X_i$  if the map of underlying topological spaces $|X|\ra \varprojlim_i |X_i|$ is a homeomorphism, and if there is an open cover of $X$ by affinoid $\Spa(A,A^+)\subset X$, such that the map
 	\[\varinjlim_{\Spa(A_i,A_i^+)\subset X_i}A_i\ra A\]has dense image, where the direct limit runs over all open affinoid
 	\[\Spa(A_i,A_i^+)\subset X_i\]over which $\Spa(A,A^+)\subset X\ra X_i$ factors.
 \end{definition}
 Sometimes we will not mention the compatible family of morphisms $f_i: X\ra X_i$ explicitly: once we write $X\sim \varprojlim_i X_i$, we mean that there exists such a family of morphisms, which is usually clear from the context.
 
 We have the following immediate lemma, which may be implicitly contained in \cite{Sch3} and \cite{SW}.
 \begin{lemma}\label{L:sub}
 	In the setting of Definition~\ref{def}, let $Y_i\subset X_i$ be a locally closed subspace for each $i$ and these $Y_i$ form a subsystem, i.e. $(Y_i)_{i\in I}$ forms a filtered inverse system under the induced transition maps such that $Y_j\simeq Y_i\times_{X_i}X_j$ for any $j\geq i$. Then there is a locally closed subspace $Y\subset X$ such that
 	\[Y\sim \varprojlim_i Y_i.\]
 Moreover, $Y\subset X$ is an open (resp. a closed) subspace if the system $Y_i$ is an open (resp. a closed) subsystem of $X_i$. In particular, if $X$ is a perfectoid space over a perfectoid field $k$ and the system $Y_i\subset X_i$ is an open subsystem of adic spaces over $k$, then $Y$ is also a perfectoid space over $k$. 
 \end{lemma}
 \begin{proof}
 	Indeed, we may assume that for any $i$, $Y_i\subset X_i$ is locally closed. Then for any $j\geq i$, we have $Y_j\simeq Y_i\times_{X_i}X_j\ra X_j$ with the composition giving the inclusion. Set $Y=Y_i\times_{X_i}X$, which does not depend on $i$. Then by Proposition 2.4.3 of \cite{SW}, we have \[Y\sim \varprojlim_i Y_i.\]
 	Since $Y_i\subset X_i$ is locally closed, the base change map $Y=Y_i\times_{X_i}X\ra X$ is a locally closed immersion. By the same argument, $Y\subset X$ is an open (resp. a closed) subspace if the system $Y_i$ is an open (resp. a closed) subsystem of $X_i$.
 \end{proof}
 
 Recall that a complete Huber pair $(A, A^+)\in k-\CAff$ with $A$ a Tate $k$-algebra is called \emph{stably uniform}, if $\mathcal{O}_X(U)$ is uniform, i.e. the subring of power-bounded elements $\mathcal{O}_X(U)^\circ$ is bounded, for all rational subsets $U\subset X=\Spa(A,A^+)$. If $(A, A^+)$ is stably uniform, then it is sheafy, i.e. $X=\Spa(A, A^+)$ is an honest adic space over $k$, cf. \cite{BV} Theorem 7.
 The following lemma is the global version of \cite{SW} Proposition 2.4.2, which asserts the existence of $X$ in Definition \ref{def} in some cases.
 \begin{lemma}\label{L:existence}
 Let $k$ be a non-archimedean field, and $(X_i)_{i\in I}$ be a filtered inverse system of adic spaces over $k$ with finite transition maps. Suppose that for any affinoid subsystem $\Spa(A_i, A_i^+)\subset X_i$, the $\varpi$-adic completion $(A, A^+)=\wh{\varinjlim_i (A_i, A_i^+)}$ is stably uniform, where $\varpi$ is a pseudo-uniformizer of $k$. Then there exists an adic space $X$ over $k$, with a compatible family of morphisms, such that $X\sim \varprojlim_i X_i$.
 \end{lemma}
 \begin{proof}
 By the assumption that $(A, A^+)=\wh{\varinjlim_i (A_i, A_i^+)}$ is stably uniform, in the above Lemma \ref{L:sub}, if each $Y_i=\Spa(B_i,B_i^+)\subset X_i$ is a rational subspace, then $Y=\Spa(B, B^+)$ with $(B,B^+)=\wh{\varinjlim_i (B_i, B_i^+)}$. Therefore, we can reduce to the affinoid case: $X_i=\Spa(A_i, A_i^+)$ with $(A_i, A_i^+) \in k-\CAff$. 
 Since we are working over a base field $k$, all the transition maps are adic. Let $X=\Spa(A, A^+)$ with 
 \[(A, A^+)=\wh{\varinjlim_i (A_i, A_i^+)},\]where the completion is taken with respect to the $\varpi$-adic topology. Then by \cite{SW} Proposition 2.4.2, $X$ satisfies the requirement.
 \end{proof}
 Let the situation be as in the above lemma. Fix an $i_0\in I$. Suppose that all the adic spaces $X_{i}$ are locally noetherian, and for any $i\geq i_0$, the transition map $X_i\ra X_{i_0}$ is \'etale. Then the presentation $X':=\varprojlim_{i\geq i_0} X_i$ defines an object in the pro-\'etale site $X_{i_0,\tr{pro\'et}}$. The adic space $X$ constructed in Lemma \ref{L:existence} is denoted by $\wh{X'}$ in \cite{Sch2} section 4.
 
 How about the uniqueness of the space $X$ in the setting of Definition \ref{def}? In general, not much is known. However, if $X$ is perfectoid, then \cite{SW} Proposition 2.4.5 says that $X$ is unique among perfectoid spaces satisfying the condition in Definition \ref{def}. In this case, $X$ is locally constructed as in the Lemma \ref{L:existence}. More precisely, we have (see also \cite{Sch5} Proposition 2.22)
 \begin{lemma}\label{L:perfectoid-local}
 Let $k$ be a perfectoid field with residue field of characteristic $p$, $(X_i)_{i\in I}$ be a filtered inverse system of adic spaces over $k$, with finite transition maps, and $X$ be a perfectoid space over $k$ such that $X\sim \varprojlim_i X_i$. Fix an $i_0\in I$. Then there exists an affinoid covering $(\Spa(A_{i_0}, A_{i_0}^+))$ of $X_{i_0}$, such that the pullbacks $\Spa(A, A^+)$  of $\Spa(A_{i_0}, A_{i_0}^+)$ under $X\ra X_{i_0}$ form a perfectoid affinoid covering of $X$ (i.e. $A$ is a perfectoid $k$-algebra). Moreover, for each $X_i\ra X_{i_0}$ with $i\geq i_0$, let $\Spa(A_i, A_i^+)$ be the pullback of the affinoid $\Spa(A_{i_0}, A_{i_0}^+)$. Then
 \[(A, A^+)\simeq \wh{\varinjlim_i (A_i, A_i^+)},\]where the completion is taken with respect to the $\varpi$-adic topology (as before, $\varpi$ is a pseudo-uniformizer of $k$), and 
 \[\Spa(A, A^+)\sim \varprojlim_i \Spa(A_i, A_i^+).\]
 \end{lemma}
 \begin{proof}
 Take any affinoid covering $(\Spa(A_{i_0}, A_{i_0}^+))$ of $X_{i_0}$, let $\Spa(A_i, A_i^+)$ be the pullback of the affinoid $\Spa(A_{i_0}, A_{i_0}^+)$, for each $X_i\ra X_{i_0}$ with $i\geq i_0$. Set  \[(A, A^+)\simeq \wh{\varinjlim_i (A_i, A_i^+)}.\] By Lemma \ref{L:sub}, $\Spa(A, A^+)\subset X$ is an open subspace, and $\Spa(A, A^+)\sim \varprojlim_i \Spa(A_i, A_i^+).$ By refining the affinoid covering $(\Spa(A_{i_0}, A_{i_0}^+))$ of $X_{i_0}$ if necessary, we get that $A$ is a perfectoid $k$-algebra. These perfectoid affinoid subspaces $\Spa(A, A^+)$ form a covering of $X$.
\end{proof}
 
 We fix a perfectoid field $k$ with residue field of characteristic $p$ throughout the rest of this section. Let $(X_i)_{i\in I}$ be a filtered inverse system of adic spaces over $k$, with finite transition maps. If the conditions in Lemma \ref{L:existence} are satisfied, we get an adic space $X$ over $k$, such that $X\sim \varprojlim_i X_i$. By Lemma \ref{L:perfectoid-local} and \cite{SW} Proposition 2.4.5, this space is the only candidate for being perfectoid. 
 
 \subsection{\'Etale quotients for finite groups actions on perfectoid spaces}
 Let $X\in \Perf_k$ be a perfectoid space, equipped with a right action over $k$ of a profinite group $G$.  We say that this action is free, if the natural map \[X\times G\ra X\times_k X,\] with the natural projection on the first factor and the group action on the second factor,  is functorially injective (i.e. a monomorphism). Here, $X\times G$ is the perfectoid space associated to (in the sense of Definition \ref{def}) $\varprojlim_{H}X\times G/H$, where $H$ runs through the set of open normal subgroups of $G$, and $X\times G/H$ is a finite disjoint union of copies of $X$. This is also the fiber product over $k$ of $X$ and (the perfectoid space associated to) $G$, see Lemma \ref{L:profinite} (and Lemma \ref{L:product}). In the above setting, let $R=X\times G\ra X\times_kX$ be the functorial equivalence relation defined by the free action of $G$ on $X$. We denote by $X/G$ the sheafification of the presheaf \[Y\mapsto X(Y)/R(Y)\] on the category $\Perf_k$ equipped with the pro-\'etale topology (cf. \cite{Sch4} 8.2). By definition, $X/G$ is then a diamond, cf. loc. cit. 8.2.8. We would like to know whether this diamond is representable. Examples in \cite{Sch4} (Proposition 10.2.6) show that, contrary to the cases of schemes (\cite{SGA}) and rigid analytic spaces (\cite{CT}), one should put additional structures to ensure that the quotients as adic (perfectoid) spaces exist, if we only require the group $G$ to be \emph{profinite}. More precisely, if $G$ is a profinite rather than finite group, then Proposition 10.2.6 of \cite{Sch4} says that there exists a finite universal homeomorphism $f: Y\ra Y'$ of two different locally noetherian adic spaces, such that $Y$ and $Y'$ have the same perfectoid profinite Galois cover $X$ with Galois group $G$, and the associated diamonds of $Y$ and $Y'$ are isomorphic via $f$. In this case, the quotient of $X$ by $G$ as a diamond exists uniquely, but both $Y$ and $Y'$ can be natural candidates of quotients of $X$ by $G$ as adic spaces. Therefore, we assume that $G$ is a \emph{finite} group, and this will be the only case that will be used later. Then it is enough to consider the category $\Perf_k$ equipped with the \'etale topology, cf. \cite{Sch1} section 7.
 The following proposition is an analogue of \cite{SGA} Th\'eor\`eme 4.1 of Expos\'e V and \cite{CT} Lemma 3.2.1.
  \begin{proposition} \label{P:quotient}
 	Let $(A,A^+)$ be an affinoid perfectoid $k$-algebra.
 	Assume that $X=\Spa(A,A^+)$ is equipped with a free right action of a finite group $G$ over $k$. Consider the induced $G$-action on $(A,A^+)$. Set \[B=A^G, \quad B^+=(A^+)^G,\] the $G$-invariant subalgebras of $A$ and $A^+$ respectively. Then the following statements hold true.
 	\begin{enumerate}
 		\item $(B,B^+)$ is a stably uniform affinoid Tate $k$-algebra. In particular, $Y=\Spa(B,B^+)$ is an analytic honest adic space over $k$. 
 		\item The \'etale sheaf associated to $Y$ on the category $\Perf_k$  is given by $X/G$ with the above notation. Moreover, the natural morphism $X\ra Y$ is a finite \'etale Galois cover with Galois group $G$.
 		\item $Y$ is a perfectoid space.
 	\end{enumerate}
 \end{proposition}
 \begin{proof}
(1) Let $A_0\subset A$ be a ring of definition. For any $g\in G$, $A_0g\subset A$ is also a ring of definition. We want to first find a ring of definition stable under the $G$-action. Let $A^\circ\subset A$ be the subring of powerbounded elements. Then $A^\circ$ is stable under the $G$-action on $A$. Indeed, by \cite{Sch1} Definition 2.6 (i), there is a subring $A_0\subset A$ such that $aA_0, a\in k^\times$ forms a basis of open neighborhoods of 0. For any $x\in A^\circ$, by definition, there exists some $a\in k^\times$ such that $\{x^n|\,n\geq0\}\subset aA_0$. Let $g\in G$ be any element. Then $aA_0g, a\in k^\times$ also forms a basis of open neighborhoods of 0, and $\{(xg)^n|\,n\geq0\}\subset aA_0g$. Therefore $\{(xg)^n|\,n\geq0\}$ is bounded, i.e. $xg\in A^\circ$. Since $A$ is perfectoid, $A^\circ$ is bounded, therefore a ring of definition. The fact that $B=A^G$ is a Tate $k$-algebra is now clear, since $aA^{\circ G}, a\in k^\times$ forms a basis of open neighborhoods of 0. By construction, we have
$B^\circ=A^{\circ G}$, which is bounded, and \[B^+=B \cap A^+=(A^+)^G \subset B^\circ\] is an open and integrally closed subring. Hence $(B,B^+)$ is an affinoid Tate $k$-algebra. 

Since $A$ is perfectoid, therefore stably uniform. Let $f_1,\dots, f_n, g\in B$ such that $(f_1,\dots, f_n)B$ is open in $B$. Let $U_1\subset Y=\Spa(B,B^+)$ be the rational subset defined by these elements. By viewing $f_1,\dots, f_n, g$ as elements of $A$ by the inclusion $B\subset A$, 
we get rational subset $U_2\subset X=\Spa(A,A^+)$. We have
 \[\mathcal{O}_Y(U_1)^\circ=B\lan \frac{f_1}{g},\cdots,\frac{f_n}{g}\ran^\circ=(A\lan \frac{f_1}{g},\cdots,\frac{f_n}{g}\ran^\circ)^G=(\mathcal{O}_X(U_2)^\circ)^G,\]
 where $B\lan \frac{f_1}{g},\cdots,\frac{f_n}{g}\ran$ and $A\lan \frac{f_1}{g},\cdots,\frac{f_n}{g}\ran$ are as in the proof of Corollary \ref{C:geometric} below.
Therefore $B\lan \frac{f_1}{g},\cdots,\frac{f_n}{g}\ran^\circ$ is bounded, and $(B, B^+)$ is stably uniform.

(2) By assumption, $X\times G\ra X\times_kX$ is a monomorphism. We view $G$ as the constant group scheme over $k$, and let $k[G]$ be the associated group algebra. Write \[Z=X\times G=\Spa(C, C^+),\quad W=X\times_k X=\Spa(D, D^+),\] where $C=A\otimes_k k[G], D=A\otimes_k A$. Then arguing in the same as \cite{EGA} Proposition 5.3.8, we get that the diagonal map $\delta_{Z/W}: Z\ra Z\times_WZ$ is an isomorphism, which implies that $C\simeq C\otimes _DC$ as complete Huber rings. This implies in turn that 
\[\Spec(A)\times G\ra \Spec(A)\times_k\Spec(A)\]is a monomorphism by \cite{EGA} Proposition 5.3.8. Therefore, by \cite{SGA} Expos\'e V, Th\'eor\`eme 4.1, $A$ is integral over $B=A^G$, and we get an \'etale quotient of the affine scheme $\Spec(A)$ by the free right action of $G$, which is given by $\Spec(B)$. Moreover, the natural map $\Spec(A)\ra \Spec(B)$ is a finite \'etale Galois cover with Galois group $G$, and we have an isomorphism \[\Spec(A)\times G\simeq \Spec(A)\times_{\Spec(B)}\Spec(A),\] i.e. an isomorphism of rings \[C=A\otimes_k k[G]\simeq A\otimes_{B}A.\] By loc. cit. $A^+$ is integral over $B^+=(A^+)^G$. Therefore, the natural morphism $X\ra Y$ is a finite \'etale Galois cover with Galois group $G$. In particular, $Y=X/G$ as \'etale sheaves on $\Perf_k$.

(3) Recall \cite{KL} Proposition 3.6.22 (a) says that, if $B\ra A$ is a morphism of uniform Banach algebras such that $A$ is perfectoid, and $A$ is faithfully finite \'etale over $B$ (i.e. the map $B\ra A$ is finite \'etale and faithfully flat), then $A$ is also perfectoid (here we have switched the notations $A$ and $B$ as opposed to those in loc. cit.). Now, $\Spa(A,A^+)$ is by assumption perfectoid, by statement (2) $B\ra A$ is faithfully finite \'etale. Thus $\Spa(B,B^+)$ is also perfectoid.
\end{proof}
The following corollary is immediate.
\begin{corollary}\label{C:quotient}
Let $X \in \Perf_k$ be a perfectoid space, equipped with a free right action of a finite group $G$. Assume that there is an affinoid perfectoid covering of $X$ with each affinoid space stable under the action of $G$. Then there is a perfectoid space $Y$ over $k$, together with a natural morphism $X\ra Y$ which is a finite \'etale Galois cover with Galois group $G$, such that as \'etale sheaves on the category $\Perf_k$,  we have the equality $Y=X/G$. 
\end{corollary}
The next corollary says that in the above situation, the \'etale quotient $Y$ is also a ``geometric quotient'', as in the case of schemes.
\begin{corollary}
\item Let the situation be as in the above corollary. The associated map of adic spaces
$\pi: X\ra Y$induces that
\begin{enumerate}\label{C:geometric}
	\item $|Y|=|X|/G$ as topological spaces,
	\item $(\mathcal{O}_Y,\mathcal{O}_Y^+)=\big((\pi_\ast\mathcal{O}_X)^G,(\pi_\ast\mathcal{O}_X^+)^G\big)$, where
	$(\pi_\ast\mathcal{O}_X)^G$ (resp. $(\pi_\ast\mathcal{O}_X^+)^G$) is the $G$-invariant subsheaf of $(\pi_\ast\mathcal{O}_X)$ (resp. $(\pi_\ast\mathcal{O}_X^+)$).
\end{enumerate} 
\end{corollary}
\begin{proof}
	Since we have the equality of \'etale shaves $Y=X/G$, the associated topological spaces satisfy $|Y|=|X/G|=|X|/G$. The second statement can be easily verified by looking at all rational subsets of $Y$. Indeed, we check the equality $(\mathcal{O}_Y,\mathcal{O}_Y^+)=\big((\pi_\ast\mathcal{O}_X)^G,(\pi_\ast\mathcal{O}_X^+)^G\big)$ on rational subsets as follows. Let $f_1,\dots, f_n, g\in B$ such that $(f_1,\dots, f_n)B$ is open in $B$. Let $U_1\subset Y=\Spa(B,B^+)$ be the rational subset defined by these elements. By viewing $f_1,\dots, f_n, g$ as elements of $A$ by the inclusion $B\subset A$, 
	we get rational subset $U_2\subset X=\Spa(A,A^+)$. By construction, we have $U_2=\pi^{-1}(U_1)$. Consider the complete Huber pair
	\[\mathcal{O}_Y(U_1)=B\lan \frac{f_1}{g},\cdots,\frac{f_n}{g}\ran, \quad \mathcal{O}_Y^+(U_1)= B\lan \frac{f_1}{g},\cdots,\frac{f_n}{g}\ran^+,\]
	where $B\lan \frac{f_1}{g},\cdots,\frac{f_n}{g}\ran$ is the completion of $B[ \frac{f_1}{g},\cdots,\frac{f_n}{g}]$, and $B\lan \frac{f_1}{g},\cdots,\frac{f_n}{g}\ran^+$ is the completion of the integral closure of $B^+[\frac{f_1}{g},\cdots,\frac{f_n}{g}]$ in $B[\frac{f_1}{g},\cdots,\frac{f_n}{g}]$. Similarly, we have the complete Huber pair
	\[\mathcal{O}_X(U_2)=A\lan \frac{f_1}{g},\cdots,\frac{f_n}{g}\ran, \quad \mathcal{O}_X^+(U_2)= A\lan \frac{f_1}{g},\cdots,\frac{f_n}{g}\ran^+.\]
	We have
	\[B\lan \frac{f_1}{g},\cdots,\frac{f_n}{g}\ran=(A\lan \frac{f_1}{g},\cdots,\frac{f_n}{g}\ran)^G, \quad B\lan \frac{f_1}{g},\cdots,\frac{f_n}{g}\ran^+=(A\lan \frac{f_1}{g},\cdots,\frac{f_n}{g}\ran^+)^G.\]Moreover, the following commutative diagram is cocartesian:
	\[\xymatrix{
		(B,B^+)\ar[r]\ar[d]&(A,A^+)\ar[d]\\
		(B\lan \frac{f_1}{g},\cdots,\frac{f_n}{g}\ran, B\lan \frac{f_1}{g},\cdots,\frac{f_n}{g}\ran^+)\ar[r]&(A\lan \frac{f_1}{g},\cdots,\frac{f_n}{g}\ran,A\lan \frac{f_1}{g},\cdots,\frac{f_n}{g}\ran^+).
	}\]
	Therefore, we get $(\mathcal{O}_Y,\mathcal{O}_Y^+)=\big((\pi_\ast\mathcal{O}_X)^G,(\pi_\ast\mathcal{O}_X^+)^G\big)$.
\end{proof}

We return to the situation as in Definition \ref{def}.
In this paper, we will always assume that the adic spaces $X_i$ in Definition \ref{def} are locally noetherian over $k$. In fact, we will only work with adic spaces associated to rigid analytic spaces over $k$, or even adic spaces associated to schemes of locally of finite type over $k$.
The following proposition will be used crucially in the next section.
\begin{proposition}\label{prop}
	Let $X$ and $(X_i)_{i\in I}$ be as in  Definition~\ref{def}, with $X\in \Perf_k$ perfectoid, $X_i\in \Adic_k^{ah}$ locally noetherian. Let $(Y_i)_{i\in I}$ be another filtered inverse system of locally noetherian adic spaces over $k$ with finite transition maps. Assume that there is an adic space $Y$ over $k$ such that $Y\sim \varprojlim_i Y_i$. Suppose that for each $i$, there exisits a finite \'etale morphism of locally noetherian adic spaces $X_i\ra Y_i$ over $k$, and these morphisms are compatible with the transition maps of the two systems.  Suppose further that there exists an $i_0\in I$ such that $X_{i_0}\ra Y_{i_0}$ is a surjection, and for any $i\geq i_0$, the morphism $X_i\ra Y_i$ is the pullback of $X_{i_0}\ra Y_{i_0}$ under the transition map $Y_i\ra Y_{i_0}$. In other words, the following diagrams are cartesian
	\[\xymatrix{X_i\ar[r]\ar[d]& X_{i_0}\ar[d]\\
		Y_i\ar[r]& Y_{i_0}
		}\]
	for all $i\geq i_0$.  Then $Y$ is perfectoid, and there exists a finite \'etale morphism $X\ra Y$, which is the pullback of $X_i\ra Y_i$ for any $i\geq i_0$.
\end{proposition}
\begin{proof}
We may assume that $X_i=\Spa(A_i, A_i^+)$, 
$X=\Spa(A, A^+)$, $Y_i=\Spa(B_i, B_i^+)$ , $Y=\Spa(B, B^+)$ are affinoid, with 
\[(A, A^+)=\wh{\varinjlim_i (A_i, A_i^+)},\quad (B, B^+)=\wh{\varinjlim_i (B_i, B_i^+)}\]
and $A$ a perfectoid $k$-algebra. By assumption, we get finite \'etale morphisms of affinoid $k$-algebras
\[(B_i, B_i^+)\ra (A_i, A_i^+),\]taking limit and $\varpi$-adic completion ($\varpi$ is a pseudo-uniformizer of $k$) we get a finite \'etale morphism
\[(B, B^+)\ra (A, A^+),\]
which is the pushforward 
\[\xymatrix{(B_i, B_i^+)\ar[r]\ar[d]& (A_i, A_i^+)\ar[d]\\
	(B, B^+)\ar[r]& (A, A^+)
}\]
for any $i\geq i_0$. By assumption, $A_i$ is a faithfully finite \'etale $B_i$-algebra (see the proof of Proposition 2.5 (3)), hence $A$ is a faithfully finite \'etale $B$-algebra. As $A$ is perfectoid, by \cite{KL} Proposition 3.6.22, $B$ is perfectoid too.
\end{proof}
In the next section, we will apply the above proposition in the case that $X_i\ra Y_i$ are finite \'etale Galois covers with Galois groups $G_i$, such that the groups $G_i$ are constant for $i$ large enough. Let $G=G_i$ for $i$ sufficiently large, then it acts freely on $X$. Let $X, (X_i), (Y_i)$ be as in the above proposition, but without the assumption of the existence of $Y$. 
\begin{corollary}\label{C:Quotient}
 There exists an affinoid perfectoid covering of $X$ with each affinoid space stable under the action of $G$. In particular, the quotient $X/G$ exists as a perfectoid space $Y$ by Corollary \ref{C:quotient}. Moreover, $Y\sim \varprojlim_i Y_i$.
\end{corollary}
\begin{proof}
	Let $\Spa(B_i, B_i^+)\subset Y_i$ be an affinoid space. Let $\Spa(A_i, A_i^+)\subset X_i$ be the inverse image of $\Spa(B_i, B_i^+)$ under the finite \'etale Galois cover $X_i\ra Y_i$. Set $(A, A^+)=\wh{\varinjlim_i (A_i, A_i^+)}$, then shrinking $\Spa(B_i, B_i^+)$ if necessary, these $\Spa(A, A^+)$ form an affinoid perfectoid covering of $X$ with each affinoid space stable under the action of $G$. Let $Y=X/G$ be the perfectoid space constructed by Corollary \ref{C:quotient}. As $G=G_i$ for $i$ sufficiently large, it is easy to see $Y\sim \varprojlim_i Y_i$.
\end{proof}

\subsection{Some useful lemmas}
The following lemma will be used.
\begin{lemma}\label{L:profinite}
	Let $X$ be a profinite set, and $k$ be a perfectoid field as above. Then there exists a perfectoid space $X^{ad}$ over $k$, such that the underlying topological space $|X^{ad}|$ is homeomorphic to $X$.
\end{lemma}
\begin{proof}
	See \cite{Sch4} Remark 8.2.2. It is given by $X^{ad}=\Spa(R, R^+)$ with $R$ (resp. $R^+$) the ring of continuous functions from $X$ to $k$ (resp. $k^\circ$).
\end{proof}
If we write $X=\varprojlim_iX_i$ with each $X_i$ finite, then the above perfectoid space $X^{ad}$ is such that $X^{ad}\sim\varprojlim_iX_i^{ad}$, where $X_i^{ad}$ is the finite perfectoid space over $k$ attached to $X_i$. Sometimes by abuse of notation, we denote also by $X$ the perfectoid space associated to the profinite set $X$.

The next lemma will not be used explicitly in the following\footnote[2]{See Remark \ref{R}.}, but it is useful in the setting of products of Shimura varieties.
\begin{lemma}\label{L:product}
Let $(X_i)_{i\in I}$ and $(Y_i)_{i\in I}$ be two filtered inverse systems of adic spaces over $k$ as in Definition \ref{def}. Let $X, Y$ be  adic spaces over $k$ such that $X\sim \varprojlim_i X_i$, $Y\sim\varprojlim_i Y_i$. Then we have
\[X\times_k Y\sim \varprojlim_i X_i\times_k Y_i.\]
\end{lemma}
\begin{proof}
	One checks directly by definition.
\end{proof}

\section{Perfectoid Shimura varieties of abelian type}
To motivate our construction, we begin with some review of the theory of geometric connected components of Shimura varieties. For a connected reductive group $G$ over $\Q$, we denote by $G^{der}$ and $G^{ad}$ the associated derived subgroup and the adjoint group respectively. The identity element of a group will be denoted by $e$ usually.

\subsection{Geometric connected components of Shimura varieties}

\subsubsection{The rational case}
As in \cite{D},  an important technique for proving results about Shimura varieties is the reduction to a problem about connected Shimura varieties. We first review some constructions in \cite{D}.

Let $(G, X)$ be a Shimura datum. We consider the associated connected Shimura varieties. So fix a connected component $X^+$ of $X$, and consider the triplet $(G^{ad}, G^{der}, X^+)$. For any open compact subgroup $K\subset G(\A_f)$, let $\Sh_K(G,X)_\C$ be the associated Shimura variety over the complex field $\C$. Consider the inverse limit \[\Sh(G,X)_\C=\varprojlim_{K}\Sh_K(G,X)_\C,\] which is a scheme (not of finite type) over $\C$. Let $\Sh^0(G,X)_\C$ be the connected component of $\Sh(G, X)_\C$ containing the image of $X^+\times\{e\}\subset X\times G(\A_f)$, which is also given by the projective limit
\[\Sh^0(G,X)_\C=\varprojlim_{K}\Sh_{K}^0(G,X)_\C,\]with $\Sh_{K}^0(G,X)_\C$ defined as the connected component of $\Sh_K(G, X)_\C$ containing the image of $X^+\times\{e\}\subset X\times G(\A_f)$, for each $K\subset G(\A_f)$. The scheme $\Sh^0(G,X)_\C$ depends only on $(G^{ad}, G^{der}, X^+)$. More precisely, let $G^{ad}(\R)^+$ be the connected component of $G^{ad}(\R)$ containing the unit for the real topology, and let $G^{ad}(\Q)^+=G^{ad}(\Q)\cap G^{ad}(\R)^+$. We define $\tau(G^{der})$ to be the linear topology on $G^{ad}(\Q)$ for which the images in $G^{ad}(\Q)$ of the congruence subgroups in $G^{der}(\Q)$ form a fundamental system of neighborhood of the identity. Then the connected Shimura variety $\Sh^0(G,X)_\C$ is given by the projective limit (cf. \cite{D} 2.1.8)
\[\Sh^0(G,X)_\C=\varprojlim_\Gamma\Gamma\setminus X^+,\]where $\Gamma$ runs through the arithmetic subgroups of $G^{ad}(\Q)^+$ which are open in $\tau(G^{der})$. Here and in the following, we use the convention as in \cite{D} 2.1.2 and \cite{Mo} 1.5: we view $\Gamma\setminus X^+$ as the algebraic varieties over $\C$ by using the complex GAGA. The system $\Big(\Gamma\setminus X^+\Big)_\Gamma$ is equipped with an action on the right of $G^{ad}(\Q)^+$ by
\[\Gamma\setminus X^+\ra \gamma^{-1}\Gamma\gamma\setminus X^+, \quad [x]\mapsto [\gamma^{-1} x]\] for any $\gamma\in G^{ad}(\Q)^+$. Therefore, the scheme $\Sh^0(G,X)_\C$ is equipped with a continuous action on the right (see the paragraph above Proposition \ref{P:delgine}) of the completion $G^{ad}(\Q)^{+\wedge}$ of $G^{ad}(\Q)^+$ for the topology $\tau(G^{der})$. We will denote this completion also as \[G^{ad}(\Q)^{+\wedge} \big(\tr{rel}. \tau(G^{der})\big)\]
 to make precise the topology. Since $\cap \Gamma=1$, where $\Gamma$ runs through arithmetic subgroups of $G^{ad}(\Q)^+$ open in the topology $\tau(G^{der})$, we can identify $G^{ad}(\Q)^+$ as a subgroup of $G^{ad}(\Q)^{+\wedge}$. Sometimes we also denote the connected Shimura variety $\Sh^0(G,X)_\C$ by $\Sh^0(G^{ad}, G^{der}, X^+)_\C$. In fact, for any triplet $(G, G', X^+)$ consisting of an adjoint group $G$, a cover $G'$ of $G$, and a $G(\R)^+$-conjugacy class of morphisms $h: \mathbb{S}\ra G_\R$ which satisfy the axioms as in the definition of a Shimura datum (cf. \cite{D} 2.1.1.1-2.1.1.3), we can define the scheme $\Sh^0(G,G',X^+)_\C$ in a similar way as above, cf. \cite{Mi1} II.1 for example.
 
 Let the notations be as above. We can recover the scheme $\Sh(G, X)_\C$ from the connected scheme $\Sh^0(G,X)_\C$. First, we recall the construction of \cite{D} 2.0.1, see also \cite{Ki} 3.3.1. Let $G$ be a group equipped with an action of a group $H$, and $\Gamma\subset G$ a $H$-stable subgroup. Suppose given an $H$-equivariant map $\varphi: \Gamma\ra H$, where $H$ acts on itself by inner automorphisms, and suppose that for $\gamma\in \Gamma$, $\varphi(\gamma)$ acts on $G$ as inner conjugation by $\gamma$. Then the elements of the form $(\gamma,\varphi(\gamma)^{-1})$ form a normal subgroup of the semi-product $G\rtimes H$. We denote 
 \[G\ast_\Gamma H\] the quotient of $G\rtimes H$ by this normal subgroup. Now we return to the previous notations. Let $Z\subset G$ be the center of the reductive group, and we denote by $\ov{Z(\Q)}$ the closure of $Z(\Q)$ in $G(\A_f)$. Let $G(\R)_+$ be the preimage of $G^{ad}(\R)^+$ under the map $G(\R)\ra G^{ad}(\R)$. Set $G(\Q)_+=G(\Q)\cap G(\R)_+, G^{der}(\Q)_+=G^{der}(\Q)\cap G(\Q)_+$. We get natural maps $G(\Q)_+\ra G^{ad}(\Q)^+, G^{der}(\Q)_+\ra G^{ad}(\Q)^+$. Following Kisin, we denote
 \[\mathcal{A}(G)=G(\A_f)/\ov{Z(\Q)}\ast_{G(\Q)_+/Z(\Q)}G^{ad}(\Q)^+.\] By \cite{D} 2.1.13, this group acts on the right on $\Sh(G, X)_\C$.
 Let $\ov{G(\Q)}_+$ be the closure of $G(\Q)_+$ in $G(\A_f)$ and set
 \[\mathcal{A}(G)^0=\ov{G(\Q)}_+/\ov{Z(\Q)}\ast_{G(\Q)_+/Z(\Q)}G^{ad}(\Q)^+.\]Then $\mathcal{A}(G)^0$ depends only on $G^{der}$. In fact, we have the equalities (cf. \cite{D} 2.1.15.1,  2.1.6.2)
 \[\begin{split} \mathcal{A}(G)^0&=G^{ad}(\Q)^{+\wedge}\big(\tr{rel}. \tau(G^{der})\big)\\
 &=\ov{G^{der}(\Q)}_+\ast_{G^{der}(\Q)_+}G^{ad}(\Q)^+\\
 &=\rho \big(G^{sc}(\A_f)\big)\ast_{\Gamma}G^{ad}(\Q)^+,
 \end{split}\]
 where $\rho: G^{sc}\ra G^{der}$ is the simply connected cover, and \[\Gamma=\rho \big(G^{sc}(\A_f)\big)\cap G^{der}(\Q)\subset G^{der}(\Q)^+.\]
 The group $\mathcal{A}(G)$ acts transitively on the set $\pi_0(\Sh(G,X)_\C)$ of connected components of $\Sh(G,X)_\C$, and the stabilizer of the component $\Sh^0(G,X)_\C$ is given by $\mathcal{A}(G)^0$. The action of $ \mathcal{A}(G)^0$ on $\Sh^0(G, X)_\C$, induced by that of $\Al(G)$ on $\Sh(G, X)_\C$, coincides with the right action described in the above paragraph. In particular, the profinite set $\pi_0(\Sh(G,X)_\C)$ is a principal homogenous space under the abelian group \[\Al(G)/\Al(G)^0= G(\A_f)/\ov{G(\Q)}_+,\] cf. \cite{D} 2.1.16.
 There is an $\Al(G)$-equivariant map \[\Sh(G,X)_\C\ra G(\A_f)/\ov{G(\Q)}_+,\] and
 the scheme $\Sh^0(G,X)_\C$ is isomorphic to the fiber at $e$ of this map. On the other hand,
 the scheme $\Sh(G,X)_\C$ can be recovered by $\Sh^0(G,X)_\C$ by some induction from $\mathcal{A}(G)^0$ to $\mathcal{A}(G)$ in the following sense:
 \begin{proposition}\label{P:induction}
 	We have the following identity
 	\[\Sh(G,X)_\C=[\Al(G)\times \Sh^0(G,X)_\C]/\Al(G)^0,\]where $\Al(G)^0$ acts on the scheme $\Al(G)\times \Sh^0(G,X)_\C$ by $(\gamma',s)\gamma=(\gamma^{-1}\gamma',s\gamma)$.
 \end{proposition}
 \begin{proof}
 	See \cite{D} Lemme 2.7.3 (cf. Proposition \ref{P:delgine} below) and the paragraph under it.
 \end{proof}

For later use, we need to consider the following situation. Let $(G, G', X^+)$ be a triplet as before, and let $G''$ be another cover of $G$ which is a quotient of $G'$. Consider the group
\[\Delta=\tr{Ker}\Big(G(\Q)^{+\wedge}\big(\tr{rel}. \tau(G')\big)\lra G(\Q)^{+\wedge}\big(\tr{rel}. \tau(G'')\big)\Big).\]Then it acts on $\Sh^0(G,G',X^+)$. This action comes from group actions on the finite levels as follows. Let $K\subset G'(\A_f)$ be an open compact subgroup. Write the covering maps by
$\phi_1: G'\ra G, \phi_2: G''\ra G, \pi: G'\ra G''$ with $\phi_1=\phi_2\circ\pi$.
Set
\[\Gamma_1=\phi_1\Big(K\cap G'(\Q)_+\Big),\quad \Gamma_2=\phi_2\Big(\pi(K)\cap G''(\Q)_+\Big).\] Then $\Gamma_1\subset \Gamma_2$, since we have
\[\pi\Big(K\cap G'(\Q)_+\Big) \subset \pi(K)\cap G''(\Q)_+.\] 
Set  \[\Delta(K)=\Gamma_2/\Gamma_1,\] which is finite. The natural map
\[\Gamma_1\setminus X^+\ra \Gamma_2\setminus X^+\] is a finite Galois cover with Galois group $\Delta(K)$. As $K\subset G'(\A_f)$ varies, these groups $\Gamma_2$ (resp. $\Gamma_1$) form a fundamental system of neighborhood of the identity for the topology $\tau(G'')$ (resp. $\tau(G')$). Moreover, these finite Galois covers $\Gamma_1\setminus X^+\ra \Gamma_2\setminus X^+ $ form an inverse system with compatible Galois groups. 
\begin{proposition} \label{P:rational-quotient}
	\begin{enumerate}
	\item We have $\Delta\simeq \varprojlim_K \Delta(K)$.
	\item $\Delta$ acts freely on $\Sh^0(G,G',X^+)_\C$ and we have 
\[\Sh^0(G,G'',X^+)_\C=\Sh^0(G,G',X^+)_\C/\Delta.\]
	\end{enumerate}
\end{proposition}
\begin{proof}
	(1) Directly by the above construction.
	
	(2) Since $\Gamma_1\setminus X^+\ra \Gamma_2\setminus X^+$ is a $\Delta(K)$-torsor, taking the limit over $K$, we see that 
	\[\Sh^0(G,G',X^+)_\C=\varprojlim_K \Gamma_1\setminus X^+\ra \Sh^0(G,G'',X^+)_\C=\varprojlim_K \Gamma_2\setminus X^+\] is a 
	$\varprojlim_K\Delta(K)$-torsor. By (1), $\Delta\simeq \varprojlim_K \Delta(K)$. In particular, the assertions in (2) hold. 
\end{proof}
\begin{remark}
The equality $\Sh^0(G,G'',X^+)_\C=\Sh^0(G,G',X^+)_\C/\Delta$ also follows from \cite{D} 2.7.11 (b).
\end{remark}
\begin{remark}\label{R:connected}
	\begin{enumerate}
\item The arithmetic subgroups of the form $\emph{Im}\Big(K\cap G'(\Q)_+\ra G(\Q)^+\Big)$ for $K\subset G'(\A_f)$ open compact defines a cofinal system in the family of all arithmetic subgroups in $G(\Q)^+$ which are open in the topology $\tau(G')$. 
\item If $(G,X)$ is a Shimura datum, for any open compact subgroup $K\subset G(\A_f)$, we have the Shimura variety $\Sh_K(G,X)_\C$. Consider 
the connected component $\Sh^0_K(G,X)_\C$. It is of the form $\Gamma\setminus X^+$, where
\[\Gamma=\emph{Im}\Big(K\cap G(\Q)_+\ra G^{ad}(\Q)^+\Big),\]which is in general larger than
\[\Gamma'=\emph{Im}\Big(K\cap G^{der}(\Q)_+\ra G^{ad}(\Q)^+\Big),\]
cf. \cite{Mi3} Remark 5.23. However, when studying the connected Shimura variety $\Sh^0(G,X)_\C=\varprojlim_K \Sh^0_K(G,X)_\C=\varprojlim_\Gamma \Gamma\setminus X^+$,
we can work with the cofinal family $\Big(\Gamma'\setminus X^+\Big)_{\Gamma'}$. This family will be more convenient in some situation, e.g. to understand the action of $\Delta$ on finite levels.
\end{enumerate}
\end{remark}

\subsubsection{$p$-integral case}\label{subsection}
In \cite{Mo} and \cite{Ki} Moonen and Kisin have adapted the construction in 3.1.1 to the setting of integral canonical models of Shimura varieties. Let us briefly review their construction in the following. Before going on, let us first make two remarks. Firstly, that the ``$p$-integral case'' in the title of this subsection refers only to the level structure, in the sense that we will only work with the Shimura varieties over $\C$ or $\C_p$, in any case only with the generic fibers. Secondly, we will appeal to the construction in this subsection for a different prime $\ell\neq p$ and it will allow us to exhibit a finite group $\Delta$ in Proposition \ref{P:delta}, which we can then quotient by as in section 2. Thus from this point of view, the prime $p$ below will be switched to $\ell$ later.

Fix a prime $p$. Assume that the reductive group $G$ is unramified at $p$ in the rest of this subsection. Let $G_{\Z_{(p)}}$ be a reductive group over $\Z_{(p)}$ with generic fiber $G$. We write $G(\Z_{(p)})$ and $G(\Z_p)$ for $G_{\Z_{(p)}}(\Z_{(p)})$ and $
G_{\Z_{(p)}}(\Z_{p})$ respectively. Consider the following scheme
\[\Sh_{G(\Z_p)}(G,X)_\C=\varprojlim_{K^p}\Sh_{G(\Z_p)K^p}(G,X)_\C,\]where $K^p$ runs through the open compact subgroups of $G(\A_f^p)$. Similarly, consider the connected component $\Sh_{G(\Z_p)}^0(G,X)_\C$ of $\Sh_{G(\Z_p)}(G,X)_\C$ containing the image of $X^+\times\{e\}\subset X\times G(\A_f)$, which is given by the projective limit
\[\Sh_{G(\Z_p)}^0(G,X)_\C=\varprojlim_{K^p}\Sh_{G(\Z_p)K^p}^0(G,X)_\C,\]where for any $K^p$, the scheme $\Sh_{G(\Z_p)K^p}^0(G,X)_\C$ is the connected component of $\Sh_{G(\Z_p)K^p}(G,X)_\C$ containing the image of $X^+\times\{e\}\subset X\times G(\A_f)$. By \cite{Mo} section 3 and \cite{Ki} 3.3, we can adapt Deligne's construction above to this setting. Namely, we have
\[\Sh_{G(\Z_p)}^0(G,X)_\C=\varprojlim_\Gamma\Gamma\setminus X^+,\]
where \[\Gamma=\tr{Im}\Big([G^{der}(\Q)_+\cap G(\Z_p)K^p]\ra G^{ad}(\Q)^+\Big)\] 
for some open compact subgroup $K^p\subset G(\A_f^p)$. Consider the group $G^{ad}(\Z_{(p)})$. Let $\tau(G^{der}_{\Z_{(p)}})$ be the linear topology on $G^{ad}(\Z_{(p)})$ having as a fundamental system of neighborhoods of the identity the images of the $\{p,\infty\}$-congruence subgroups $G^{der}(\Z_{(p)})\cap K^p$, where $K^p$ is an open compact subgroup of $G(\A_f^p)$. Then in the above projective limit \[\Sh_{G(\Z_p)}^0(G,X)_\C=\varprojlim_\Gamma\Gamma\setminus X^+,\]
$\Gamma$ runs through the $\{p,\infty\}$-arithmetic subgroups of $G^{ad}(\Z_{(p)})$ which are open in $\tau(G^{der}_{\Z_{(p)}})$. On $\Sh_{G(\Z_p)}^0(G,X)_\C$ we have a continuous action of $G^{ad}(\Z_{(p)})^{+\wedge}$, the completion of $G^{ad}(\Z_{(p)})^+$ relative to $\tau(G^{der}_{\Z_{(p)}})$. If we need to specify the topology, we will denote this group by
\[G^{ad}(\Z_{(p)})^{+\wedge}\big(\tr{rel}. \tau(G^{der}_{\Z_{(p)}})\big).\]

Let $Z_{\Z_{(p)}}$ be the center of $G_{\Z_{(p)}}$. Write $G(\Z_{(p)})_+=G(\Z_{(p)})\cap G(\Q)_+$ and let $\ov{Z(\Z_{(p)})}$ be the closure of $Z(\Z_{(p)})$ in $Z(\A_f^p)$ and $\ov{G(\Z_{(p)})}_+$ be the closure of $G(\Z_{(p)})_+$ in $G(\A_f^p)$. Set
\[\Al(G_{\Z_{(p)}})=G(\A_f^p)/\ov{Z(\Z_{(p)})}\ast_{G(\Z_{(p)})_+/Z(\Z_{(p)})}G^{ad}(\Z_{(p)})^+\]
and \[\Al(G_{\Z_{(p)}})^0=\ov{G(\Z_{(p)})}_+/\ov{Z(\Z_{(p)})}\ast_{G(\Z_{(p)})_+/Z(\Z_{(p)})}G^{ad}(\Z_{(p)})^+.\]The latter depends only on $G^{der}_{\Z_{(p)}}$, and it is just the completion $G^{ad}(\Z_{(p)})^{+\wedge} \big(\tr{rel}. \tau(G^{der}_{\Z_{(p)}})\big)$. We can recover the scheme $\Sh_{G(\Z_p)}(G,X)_\C$ by $\Sh_{G(\Z_p)}^0(G,X)_\C$ by induction from $\Al(G_{\Z_{(p)}})^0$ to $\Al(G_{\Z_{(p)}})$ (cf. \cite{Ki} Proposition 3.3.10)
\[\Sh_{G(\Z_p)}(G,X)_\C=[\Al(G_{\Z_{(p)}})\times \Sh_{G(\Z_p)}^0(G,X)_\C]/\Al(G_{\Z_{(p)}})^0.\]
By \cite{Ki} Lemma 3.3.3, we have
\[\Al(G_{\Z_{(p)}})/\Al(G_{\Z_{(p)}})^0=\Al(G)^0\setminus \Al(G)/G(\Z_p)= \ov{G(\Q)}_+\setminus G(\A_f)/G(\Z_p).\]
As before, there is an $\Al(G_{\Z_{(p)}})$-equivariant map \[\Sh_{G(\Z_p)}(G,X)_\C\ra \Al(G_{\Z_{(p)}})/\Al(G_{\Z_{(p)}})^0,\] and the scheme $\Sh_{G(\Z_p)}^0(G,X)_\C$ is isomorphic to the fiber at $e$ of this map.

Clearly, for any triplet $(\mathcal{G}, \mathcal{G}', X^+)$ consisting of an adjoint group $\mathcal{G}$ over $\Z_{(p)}$, a cover $\mathcal{G}'$ of $\mathcal{G}$, and a $G(\R)^+$-conjugacy class ($G:=\mathcal{G}_\Q$) of morphisms $h: \mathbb{S}\ra G_\R$ which satisfy the axioms as in the definition of a Shimura datum, we can define the scheme $\Sh^0(\mathcal{G},\mathcal{G}',X^+)_\C$  in the same way. Let $(\mathcal{G}, \mathcal{G}', X^+)$ be such a triplet, and let $\mathcal{G}'\ra \mathcal{G}'' $ be a central isogeny for another cover $\mathcal{G}''$ of $\mathcal{G}$. Consider the group
\[\Delta=\tr{Ker}\Big(\mathcal{G}(\Z_{(p)})^{+\wedge}\big(\tr{rel}. \tau(\mathcal{G}')\big)\lra \mathcal{G}(\Z_{(p)})^{+\wedge}\big(\tr{rel}. \tau(\mathcal{G}'')\big)\Big),\]which is a \emph{finite} group by \cite{Ki} 3.3.9 or \cite{Mo} 3.21.1. This group
$\Delta$ acts freely on $\Sh^0(\mathcal{G},\mathcal{G}',X^+)_\C$ and we have (cf. \cite{Mo} 3.21.1)
\[\Sh^0(\mathcal{G},\mathcal{G}'',X^+)_\C=\Sh^0(\mathcal{G},\mathcal{G}',X^+)_\C/\Delta.\]This quotient can be understood on finite levels as in the paragraph above Proposition \ref{P:rational-quotient}.

\subsection{The algebraic construction}\label{S}
Let us first review the definition of a Shimura datum of abelian type.
\begin{definition}\label{D:abel}
	\begin{enumerate}
		\item A Shimura datum $(G, X)$ is called of Hodge type, if there exists an embedding into a Siegel Shimura datum $(G,X)\hookrightarrow (\GSp_{2g}, S^{\pm})$.
		\item A Shimura datum $(G,X)$ is called of abelian type, if there is a Shimura datum $(G_1,X_1)$ of Hodge type,  together with a central isogeny $G_1^{der}\ra G^{der}$ which induces an isomorphism between the associated adjoint Shimura data $(G_1^{ad}, X_1^{ad})\st{\sim}{\ra}(G^{ad}, X^{ad})$.
		\end{enumerate}
\end{definition}
If $(G,X)$ is a Shimura datum of abelian type, the associated Shimura varieties $\Sh_K(G,X)$ are called Shimura varieties of abelian type. Note that in (2) of the above definition,  we have an isomorphism of Hermitian symmetric domains  $X^+\simeq X_1^+$. The simple factors in $G^{ad}$ for Shimura data of abelian type are classified in \cite{D} 2.3.8. It includes simple groups of types A, B, C, and most parts of type D, cf. loc. cit. for more details. Among all Shimura varieties, those of abelian type form an important and in fact the main class. These varieties were studied in many places, for example see \cite{D,Mi1, Ki}. 
\begin{remark}\label{R:abelian}
	If $(G,X)$ is a Shimura datum of abelian type, by Lemma 3.4.13 of \cite{Ki}, we can find a choice of $(G_1^\sharp, X_1^\sharp)$ satisfying the above definition, which has the maximal derived subgroup, among all of such Hodge type datum $(G_1,X_1)$, in the sense that the central isogeny
	$G_1^{\sharp der}\ra G^{der}$ factors through $G_1^{der}\ra G^{der}$, and makes $G_1^{\sharp der}$ into a cover of $G_1^{der}$.
\end{remark}

The class of abelian type Shimura varieties is strictly larger than the class of Hodge type Shimura varieties. Here are some immediate examples of abelian type Shimura varieties which are not of Hodge type. We will give some further examples in the next section.
\begin{example}\label{E:quaternion}
	Let $F$ be a totally real field and assume that $F\neq \Q$. Let $D$ be a quaternion algebra over $F$. For each open compact subgroup $K\subset D^\times(\A_{F,f})$, we have the associated quaternionic Shimura variety $\Sh_K$ of level $K$, which is of abelian type by Degline's classification. These varieties $\Sh_K$ are of PEL type if and only if one of the following two cases holds:
	\begin{itemize}
		\item $D\simeq M_2(F)$ is split, i.e. the Hilbert modular case; 
		\item $D$ is a totally indefinite quaternion division algebra, i.e. for all $v|\infty$, $D_v\simeq M_2(\R)$, and the associated Shimura varieties are compact.
	\end{itemize}
	We claim that if the quaternionic Shimura varieties $\Sh_K$ are not of PEL type (e.g. the case of Shimura curves), then they are neither of Hodge type. Indeed, for these $\Sh_K$ excluded from the above two cases, the weight morphisms $w$ are not defined over $\Q$: for any Shimura variety $\Sh_K$ associated to $D$, the weight morphism $w$ is defined over the subfield of $\ov{\Q}$ fixed by the automorphisms of $\ov{\Q}$ stabilizing the set of archimedean places $I_{nc}=\{v|\infty\,|\, D_v\simeq M_2(\R)\}$. This field of definition equals $\Q$ if and only $I_{nc}=\{v|\infty\}=\mathrm{Hom}(F,\R)$, see \cite{Mi3} Example 5.24.
\end{example}

Fix a prime $p$ from now on. In this subsection, we continue to study geometric connected components of Shimura varieties, but with the level $K^p$ outside $p$ fixed.
Let $(G,X)$ be a Shimura datum. Consider the following schemes
\[\Sh_{K^p}(G, X)_{\C}=\varprojlim_{K_p}\Sh_{K_pK^p}(G, X)_{\C}\]
and its connected component\[\Sh_{K^p}^0(G,X)_{\C}=\varprojlim_{K_p}\Sh_{K_pK^p}^0(G,X)_{\C},\]where in both limit $K_p$ runs through open compact subgroups of $G(\Q_p)$.
The scheme $\Sh_{K^p}^0(G, X)_\C$ can be described in a similar way to those in the last subsection. More precisely, we have
\[ \Sh_{K^p}^0(G,X)_\C=\varprojlim_\Gamma\Gamma\setminus X^+,\]where $\Gamma$ runs through the arithmetic subgroups of $G^{ad}(\Q)^+$ in the form of \[\tr{Im}\Big([G^{der}(\Q)_+\cap K_pK^p]\ra G^{ad}(\Q)^+\Big),\] for some open compact subgroup $K_p\subset G(\Q_p)$, cf. Remark \ref{R:connected} (2). Consider the group
\[\Gamma_0=\tr{Im}\Big([G^{der}(\Q)_+\cap K^p]\ra G^{ad}(\Q)^+\Big).\]This is a $(p,\infty)$-arithmetic subgroup of $G^{ad}(\Q)^+$. It acts on the system of varieties \[\Big(\Sh_{K_pK^p}^0(G,X)_{\C}\Big)_{K_p}=\Big(\Gamma\setminus X^+\Big)_\Gamma,\]where $\Gamma=\tr{Im}\Big([G^{der}(\Q)_+\cap K_pK^p]\ra G^{ad}(\Q)^+\Big)$ for some open compact subgroup $K_p\subset G(\Q_p)$.
There is a linear topology $\tau(G^{der})$ on $\Gamma_0$, for which the above subgroups 
$\Gamma$ of $\Gamma_0$ form a fundamental system of neighborhood of the identity. Let $\Gamma_0^\wedge$ be the completion of $\Gamma_0$ with respective to this topology. As before, when we need to specify the topology, we will denote it by
\[\Gamma_0^\wedge \big(\tr{rel}. \tau(G^{der})\big).\]
 Then $\Gamma_0^\wedge (\tr{rel}. \tau(G^{der}))$ acts continuously on the scheme $\Sh_{K^p}^0(G,X)_\C$. On the other hand, the group $G(\Q_p)$ acts on the scheme $\Sh_{K^p}(G, X)_\C$ by $p$-adic Hecke correspondences. Let $G_0$ be the subgroup of $G(\Q_p)$ which stabilizes the subscheme $\Sh_{K^p}^0(G, X)_\C\subset \Sh_{K^p}(G, X)_\C$. Set
\[\Al(K^p)=G(\Q_p)\ast_{K^p\cap G^{der}(\Q)_+}\Gamma_0\]
and \[\Al(K^p)^0= G_0\ast_{G_0\cap K^p\cap G^{der}(\Q)_+}\Gamma_0.\] We have the following lemma.
\begin{lemma}
We have $\Al(K^p)^0=\Gamma_0^\wedge \big(\tr{rel}. \tau(G^{der})\big)$. In particular, $\Al(K^p)^0$ depends only on $G^{der}$ and not on $G$.
\end{lemma}
\begin{proof}
 See \cite{D} 2.1.15.1.
\end{proof}
There is an $\Al(K^p)$-equivariant map \[\Sh_{K^p}(G, X)_{\C}\ra \Al(K^p)/\Al(K^p)^0=G(\Q_p)/G_0,\] and $\Sh_{K^p}^0(G, X)_{\C}$ is isomorphic to the fiber at $e$ of this map. Conversely, we can construct $\Sh_{K^p}(G, X)_{\C}$ from $\Sh_{K^p}^0(G, X)_{\C}$ by an induction.
\begin{lemma}
There is an isomorphism of schemes with continuous $G(\Q_p)$-action
\[\Sh_{K^p}(G, X)_{\C}\st{\sim}{\lra}[\Al(K^p)\times \Sh_{K^p}^0(G, X)_{\C}]/\Al(K^p)^0.\]
\end{lemma}
\begin{proof}
	This is a special case of \cite{D} Lemme 2.7.3 (cf. Proposition \ref{P:delgine} below), see also Proposition \ref{P:induction}.
\end{proof}

Assume that we have two Shimura data $(G_1, X_1)$ and $(G_2, X_2)$, such that there is a central isogeny $G_1^{der}\ra G^{der}_2$ which induces an isomorphism $(G_1^{ad}, X_1^{ad})\st{\sim}{\ra}(G_2^{ad}, X_2^{ad})$. We will not assume that $(G_2, X_2)$ is of abelian type, unless clearly stated (in the next subsection, after Remark \ref{R}). We fix an open compact subgroup $K_2^p\subset G_2(\A_f^p)$, then we can and we do fix an open compact subgroup $K_1^p\subset G_1(\A_f^p)$ such that $K^p_1\cap G_1^{der}(\A_f^p)$ maps to $K_2^p\cap G_2^{der}(\A_f^p)$ under the above isogeny. 
Set \[\Delta=\varprojlim_{K_{1p}}\frac{\tr{Im}\Big([G^{der}_2(\Q)_+\cap \pi(K_{1p}^{der})K_2^p]\ra G^{ad}_2(\Q)^+\Big)}{\tr{Im}\Big([G^{der}_1(\Q)_+\cap K_{1p}K_1^p]\ra G^{ad}_1(\Q)^+\Big)},\]
where $\pi: G_1^{der}\ra G^{der}_2$ is the central isogeny, $K_{1p}$ runs through open compact subgroups of $G_1(\Q_p)$, and $K_{1p}^{der}=K_{1p}\cap G_1^{der}(\Q_p)$. By definition, the group $\Delta$ depends on the choices of $K_1^p$ and $K_2^p$. It acts on $\Sh_{K^p_1}^0(G_1, X_1)_{\C}$, and this action comes from the Hecke action, see the last paragraph of this subsection.
We have the following proposition.
\begin{proposition}\label{P:delta}
$\Delta$ is a finite group, and it acts freely on $\Sh_{K^p_1}^0(G_1, X_1)_{\C}$. Moreover, there is an isomorphism of schemes with continuous $\Al(K^p_2)^0$-action
\[\Sh_{K^p_2}^0(G_2, X_2)_{\C}\st{\sim}{\lra}\Sh_{K^p_1}^0(G_1, X_1)_{\C}/\Delta.\]
\end{proposition}
\begin{proof}
We need	only to prove that $\Delta$ is a finite group. Take a prime $\ell\neq p$ such that $K^p_1$ (and $K^p$) is  hyperspecial at $\ell$. Let $\wt{\Delta}$ be the finite group as in \ref{subsection} for $p=\ell$. We can write $\wt{\Delta}$ as an inverse limit, similar to the descriptions in 3.1.1 and the definition of $\Delta$ above. Then we see that there is a surjection $\wt{\Delta}\ra \Delta$.
The remaining statements can be deduced in the same way as Proposition \ref{P:rational-quotient}. 
\end{proof}
\begin{remark}
	Given $K_2^p$, if we can choose some $K_1^p$ such that \[\emph{Im}\Big([G^{der}_2(\Q)_+\cap K_2^p]\ra G^{ad}_2(\Q)^+\Big)=\emph{Im}\Big([G^{der}_1(\Q)_+\cap K_1^p]\ra G^{ad}_1(\Q)^+\Big)=:\Gamma_0,\] then as in 3.1.1 and 3.1.2, \[\Delta=\emph{Ker}\Big(\Al(K^p_1)^0\ra \Al(K^p_2)^0\Big)=\emph{Ker}\Big(\Gamma_0^\wedge\big(\tr{rel}. \tau(G_1^{der})\big)\lra \Gamma_0^\wedge\big(\tr{rel}. \tau(G_2^{der})\big)\Big).\]
\end{remark}

The above constructions are true over any algebraically closed field which contains the reflex fields of the involved Shimura datum. Here we work over algebraically closed field to make use of geometric connected components of our Shimura varieties. One can of course work over large enough (so that the geometric connected components of Shimura varieties are defined) perfectoid field as the base field. Fix an algebraic closure $\ov{\Q}_p$ of $\Q_p$, and let $\C_p$ be its $p$-adic completion. From now on, we work over the $p$-adic field $\C_p$. Fix an isomorphism of fields $\C\simeq \ov{\Q}_p$. We can base change schemes over $\C$ to $\C_p$ via the composition of this morphism and $\ov{\Q}_p\hookrightarrow\C_p$. In particular, we consider the associated Shimura varieties over $\C_p$. Then with the above notations, we have \[\Sh_{K^p_2}^0(G_2, X_2)_{\C_p}\st{\sim}{\lra}\Sh_{K^p_1}^0(G_1, X_1)_{\C_p}/\Delta\] and
\[\Sh_{K^p_2}(G_2, X_2)_{\C_p}\st{\sim}{\lra}[\Al(K^p_2)\times \Sh_{K^p_2}^0(G_2, X_2)_{\C_p}]/\Al(K^p_2)^0.\]

We would like to briefly recall that how the above quotient is at a finite level, see the paragraph above Proposition \ref{P:rational-quotient} and Remark \ref{R:connected}. For simplicity we work over $\C$. For an open compact subgroup $K_{1p}\subset G_1(\Q_p)$, we have a map of connected Shimura varieties
\[ \Gamma_1\setminus X^+\ra \Gamma_2\setminus X^+,\]which is a finite \'etale Galois cover with Galois group denoted by
\[\Delta(K_{1p})=\Gamma_2/\Gamma_1, \]where $X^+=X^+_1=X^+_2$, and
\[\Gamma_1=\phi_1\Big(K_{1p}K^p_1\cap G_1^{der}(\Q)_+\Big),\quad \Gamma_2=\phi_2\Big(\pi\big(K_{1p}^{der})K^p_2\cap G_2^{der}(\Q)_+\Big).\]
Here as in the last subsection, $\phi_1: G_1^{der}\ra G^{ad}_1, \phi_2: G_2^{der}\ra G^{ad}_2, \pi: G_1^{der}\ra G_2^{der}$ are the covering maps, with $\phi_1=\phi_2\circ\pi$.
For $K_{1p}'\subset K_{1p}$, we have a commutative diagram of morphisms between connected Shimura varieties
\[\xymatrix{\Gamma_1'\setminus X^+\ar[d]\ar[r]& \Gamma_1\setminus X^+\ar[d]\\
	\Gamma_2'\setminus X^+\ar[r]& \Gamma_2\setminus X^+,
	} \]
	which is compatible with the map of groups $\Delta(K_{1p}')\ra \Delta(K_{1p})$. Taking inverse limit over all the open compact subgroups $K_{1p}\subset G_1(\Q_p)$, we get back the finite \'etale Galois cover
	\[ \Sh_{K^p_1}^0(G_1, X_1)\ra \Sh_{K^p_2}^0(G_2, X_2)\]
	with Galois group \[\Delta=\varprojlim_{K_{1p}}\Delta(K_{1p}).\]
	By Proposition \ref{P:delta}, $\Delta$ is finite. Hence, there exists some $K_{1p}'$ such that for all $K_{1p}\subset K_{1p}'$, we have $\Delta=\Delta(K_{1p})=\Delta(K_{1p}')$.

\subsection{The perfectoid construction}
In this subsection, we will always work on schemes or adic spaces over the perfectoid field $\C_p$, so we omit the subscript of the base field $\C_p$ from now on. Since we are working with the generic fiber, as in the last subsection, we will not assume that the reductive group $G$ of a Shimura datum $(G, X)$ is unramified at $p$. Sometimes it is convenient to fix a maximal open compact subgroup $G(\Z_p)$ of $G(\Q_p)$, which means that we fix some suitable integral model $G_{\Z_p}$ of $G$ over $\Z_p$ and take $G(\Z_p)=G_{\Z_p}(\Z_p)$.

Let $(G,X)$ be a Shimura datum. For $K^p\subset G(\A_f^p)$ fixed, the scheme
\[\Sh_{K^p}(G, X)=\varprojlim_{K_p}\Sh_{K_pK^p}(G,X)\]
forms a pro-finite \'etale cover of the Shimura variety $\Sh_{G(\Z_p)K^p}(G,X)$ over $\C_p$. In fact,
\[\Sh_{K^p}(G, X)\lra \Sh_{G(\Z_p)K^p}(G,X)\]is a $G(\Z_p)$-torsor for the pro-\'etale topology on $\Sh_{G(\Z_p)K^p}(G,X)$. If we use the language of pro-\'etale topology as developed in \cite{Sch2}, \[\Sh_{K^p}(G, X)=\varprojlim_{K_p}\Sh_{K_pK^p}(G,X)\]is an object in the pro-\'etale site $\Sh_{G(\Z_p)K^p}(G,X)_{\tr{pro\'et}}$. In fact, it is in the pro-finite \'etale site $\Sh_{G(\Z_p)K^p}(G,X)_{\tr{prof\'et}}$. We use the fully faithful embedding of categories \[\Sh_{G(\Z_p)K^p}(G,X)_{\tr{prof\'et}}\subset \Sh_{G(\Z_p)K^p}(G,X)_{\tr{pro\'et}}\] to view it as an object of the pro-\'etale site. 

For any open compact subgroup $K_p\subset G(\Q_p)$, let $\Sh_{K_pK^p}(G,X)^{ad}$ be the associated adic space of $\Sh_{K_pK^p}(G,X)$ over $\C_p$. We have the following theorem of Scholze.
\begin{theorem}[\cite{Sch3},Theorem IV. 1.1]\label{thm} 
If $(G,X)$ is a Hodge type Shimura datum, then there is a perfectoid space $S_{K^p}$ over $\C_p$ such that
\[S_{K^p}\sim \varprojlim_{K_p}\Sh_{K_pK^p}(G,X)^{ad},\]
where $\sim$ is as in Definition \ref{def}.
\end{theorem}

We keep to assume that $(G,X)$ is of Hodge type. Scholze's theorem says that, the object \[\Sh_{K^p}\in \Sh_{G(\Z_p)K^p}(G,X)^{ad}_{\tr{pro\'et}}\] with the pro-\'etale presentation $\Sh_{K^p}=\varprojlim_{K_p}\Sh_{K_pK^p}(G,X)^{ad}$ is perfectoid, and $S_{K^p}$ is the associated perfectoid space, cf. \cite{Sch2} Definition 4.3 (ii). Since there is a morphism of ringed spaces \[\Sh_{K_pK^p}(G,X)^{ad}\ra \Sh_{K_pK^p}(G,X)\] by construction, passing to limits over $K_p$, we get a map on the underlying topological spaces
\[|S_{K^p}|=\varprojlim_{K_p}|\Sh_{K_pK^p}(G,X)^{ad}|\ra |\Sh_{K^p}(G, X)|=\varprojlim_{K_p}|\Sh_{K_pK^p}(G,X)|.\]
We have a natural map of sites \[\Sh_{G(\Z_p)K^p}(G,X)^{ad}_{\tr{pro\'et}}\lra \Sh_{G(\Z_p)K^p}(G,X)_{\tr{pro\'et}},\]given by the analytification functor. Clearly, 
$\Sh_{K^p}$ is sent to $\Sh_{K^p}(G, X)=\varprojlim_{K_p}\Sh_{K_pK^p}(G,X)$ under the map \[\Sh_{G(\Z_p)K^p}(G,X)^{ad}_{\tr{pro\'et}}\lra \Sh_{G(\Z_p)K^p}(G,X)_{\tr{pro\'et}}.\]

Before proceeding further, we would like to discuss the perfectoid version of \cite{D} 2.7.1-2.7.3. Let us recall the setting, for which we have already seen several examples given by Shimura varieties in the previous subsections.  Let $\Gamma$ be a locally profinite group, and consider a projective system $(S_K)_K$ of schemes, indexed by the open compact subgroups $K$ of $\Gamma$. Suppose that for each $g\in \Gamma$, there is a given isomorphism of schemes
\[\rho_K(g): S_K\ra S_{gKg^{-1}},\]
such that 
\begin{itemize}
	\item $\rho_K(k)=id, \,\forall\, k\in K$,
	\item if $L\vartriangleleft K$ is a normal subgroup, then those $\rho_L(k)$ define a right action of $K/L$ on $S_L$, and we have
	\[S_L/(K/L)\simeq S_K.\]
\end{itemize}
Let $S=\varprojlim_K S_K$. We call it a scheme equipped with a continuous right action of $\Gamma$. Let $\pi$ be a profinite set, equipped with a continuous right action of $\Gamma$. We assume that this action is transitive, and the orbits of an open compact subgroup are open. Fix a point $e\in\pi$. Let $\Delta$ be the stabilizer of $e$ in $\Gamma$. Then we have a homeomorphism $\Gamma/\Delta\ra \pi$ induced by the group action. Let $S$ be a scheme equipped with a continuous right action of $\Gamma$, and a continuous $\Gamma$-equivariant map $S\ra \pi$. Let $S_e$ the fiber of $S$ over $e$. It is equipped with a continuous right action of $\Delta$: for $K\subset \Gamma$ open compact, $S_{e}/(K\cap \Delta)$ is the fiber over $eK$ of $S/K\ra \pi/K$, and \[S_e= \varprojlim_K S_e/(K\cap\Delta).\]
\begin{proposition}[\cite{D} Lemme 2.7.3]\label{P:delgine}
 $S\mapsto S_e$ is an equivalence of the category of schemes equipped with a continuous right action of $\Gamma$ and a continuous equivariant map to $\pi$, and the category of schemes equipped with a continuous right action of $\Delta$.
\end{proposition}
For $S$ and $S_e$ as in the above proposition, we have
\[S=[\Gamma\times S_e]/\Delta,\]where $\Delta$ acts on $\Gamma\times S_e$ by $(\gamma',s)\gamma=(\gamma^{-1}\gamma', s\gamma)$.

Let $k$ be a perfectoid field, and $S$ be a scheme equipped with a continuous right action of $\Gamma$ over $k$, such that for any $K$, the scheme $S_K$ is locally of finite type over $k$. Let $\pi, e, S_e$ and $\Delta$ be as above. Let $S_{K}^{ad}, (S_e/(K\cap\Delta))^{ad}$ be the adic spaces over $k$ associated to $S_K$ and $S_e/(K\cap\Delta)$ respectively.
\begin{proposition}\label{P:connected}
	Assume that there exists a perfectoid space $\S$ over $k$ such that  $\S\sim \varprojlim_K S_K^{ad}$.
	\begin{enumerate}
		\item There is an induced action of $\Gamma$ on $\mathcal{S}$.
		\item There is a $\Gamma$-equivariant map $\mathcal{S}\ra \pi$. If we view $\pi$ as a perfectoid space over $k$ (cf. Lemma \ref{L:profinite}). Then the map $\mathcal{S}\ra \pi$ is a morphism of perfectoid spaces.
		\item Let $\mathcal{S}_e$ be the fiber over $e$ of this morphism. Then $\mathcal{S}_e$ is equipped with a right action of $\Delta$, and we have 
		\[\S_e\sim\varprojlim_K (S_e/(K\cap\Delta))^{ad}.\]In particular, $\S_e$ is perfectoid. Moreover, we have an isomorphism of perfectoid spaces \[\mathcal{S}\simeq [\Gamma\times \mathcal{S}_e]/\Delta.\] 
	\end{enumerate}
\end{proposition}
\begin{proof}
	(1) For each $g\in \Gamma$, by GAGA, we get an isomorphism of adic spaces
	\[\rho_K(g): S_K^{ad}\ra S_{gKg^{-1}}^{ad},\]satisfying the similar properties to the case of schemes. By assumption,  we have a compatible family of morphisms $f_K: \mathcal{S}\ra S_{K}^{ad}$. Then, the composition morphisms $\rho_K(g)\circ f_K$ form a compatible family. Since $\mathcal{S}$ is perfectoid, by \cite{SW} Proposition 2.4.5, there is a unique morphism $g: \mathcal{S}\ra \mathcal{S}$, such that $\rho_K(g)\circ f_K=f_K\circ g$. In this way we get a right action of $\Gamma$ on $\mathcal{S}$.
	
	(2) and (3): For each $K$, we get an induced map $S_K=S/K\ra \pi/K$. View $\pi/K$ as a finite scheme over $k$. Then passing to the adic spaces, we get $S_K^{ad}\ra (\pi/K)^{ad}$. Composing this map with $f_K: \S\ra S_K^{ad}$, we get a compatible family of morphisms
	\[\S\ra (\pi/K)^{ad}.\]Note the underlying topological spaces of $\pi/K$ and $(\pi/K)^{ad}$ are the same.
	If we view $\pi$ as a perfectoid space over $k$, then $\pi\sim\varprojlim_K(\pi/K)^{ad}$. Again, by applying Proposition 2.4.5 of \cite{SW}, we get a morphism of perfectoid spaces \[\S\ra\pi,\]which is easily to seen $\Gamma$-equivariant. If we consider the fiber $\S_e$ over $e\in\pi$ of this morphism, then by construction, $\S_e$ is $\Delta$-invariant under the right action of $\Gamma$ on $\S$. To show that
	\[\S_e\sim \varprojlim_K(S_e/(K\cap \Delta))^{ad},\] we note first that the image of the composition
	\[\S_e\hookrightarrow \S \ra S_{K}^{ad}\ra (\pi/K)^{ad}\] is $eK$. Thus $f_K$ induces $\S_e\ra (S_e/(K\cap \Delta))^{ad}$. Moreover, the following diagram is catesian
	\[\xymatrix{\S_e\ar[d]\ar[r]&\S\ar[d]\\
	(S_e/(K\cap \Delta))^{ad}\ar[r]&S_{K}^{ad}.}\]Thus we can conclude, cf. Proposition \ref{L:sub}. Finally, the isomorphism 
\[\mathcal{S}\simeq [\Gamma\times \mathcal{S}_e]/\Delta\] is deduced by the isomorphism 
\[\S\simeq \S_e\times \pi=\S_e\times (\Gamma/\Delta).\]
	\end{proof}

We apply the above construction to the perfectoid Shimura variety $S_{K^p}$ in Theorem \ref{thm}.
\begin{corollary}\label{C:connected}
Let $(G,X)$ be a Hodge type Shimura datum and $K^p\subset G(\A_f^p)$ be a sufficiently small open compact subgroup.
\begin{enumerate}
\item There is an action of $\mathcal{A}(K^p)$ on the perfectoid space $S_{K^p}$, and a $\mathcal{A}(K^p)$-equivariant morphism of perfectoid spaces \[S_{K^p}\ra \mathcal{A}(K^p)/\mathcal{A}(K^p)^0.\]
\item There is an open and closed perfectoid subspace $S^0_{K^p}\subset S_{K^p}$ over $\C_p$ such that
\[S^0_{K^p}\sim \varprojlim_{K_p}\Sh_{K_pK^p}^0(G,X)^{ad}.\]The subspace $S^0_{K^p}$ is stable under the action of the subgroup $\mathcal{A}(K^p)^0$ of $\mathcal{A}(K^p)$. Moreover, we have an isomorphism of perfectoid spaces
\[S_{K^p}\simeq [\mathcal{A}(K^p)\times S^0_{K^p}]/\mathcal{A}(K^p)^0.\]
\end{enumerate}
\end{corollary}

As the case of $S_{K^p}$, we have an explanation of $S_{K^p}^0$ in the above corollary by using pro-\'etale sites as follows.
Let \[\Sh^0_{K^p}\in \Sh^0_{G(\Z_p)K^p}(G,X)^{ad}_{\tr{pro\'et}}\] be the object with pro-\'etale presentation $\Sh^0_{K^p}=\varprojlim_{K_p}\Sh^0_{K_pK^p}(G,X)^{ad}$. Then it is perfectoid, and the associated perfectoid space $\wh{\Sh^0_{K^p}}$ is exactly $S^0_{K^p}$.
Similar to the case of $S_{K^p}$, $\Sh^0_{K^p}$ is sent to $\Sh^0_{K^p}(G, X)=\varprojlim_{K_p}\Sh^0_{K_pK^p}(G,X)$ under the map of sites \[\Sh^0_{G(\Z_p)K^p}(G,X)^{ad}_{\tr{pro\'et}}\lra \Sh^0_{G(\Z_p)K^p}(G,X)_{\tr{pro\'et}}.\]

We will need a construction from the inverse direction as in Proposition \ref{P:connected}. 
\begin{proposition}\label{P:perf-induction}
	Let the setting be as in the paragraph above Proposition \ref{P:connected}. Assume that there exists a perfectoid space $\S_e$ over $k$ such that $\S_e\sim\varprojlim_K (S_e/(K\cap\Delta))^{ad}$. Then there exist a right action of $\Delta$ on $\S_e$, and a perfectoid space $\mathcal{S}$ over $k$, equipped an action of $\Gamma$ and a $\Gamma$-equivariant morphism of perfectoid spaces
	\[\mathcal{S}\ra \pi,\] such that $\S\sim \varprojlim_K S_K^{ad}$ and $\S_e$ is isomorphic to the fiber at $e$ of this morphism.
	\end{proposition}
\begin{proof}
	The fact that $\Delta$ acts on $\S_e$ follows by the same argument in the proof of (1) of Proposition \ref{P:connected}.
	We consider the adic space
	\[\S:=\S_e\times_k \pi=\S_e\times_k\Gamma/\Delta.\]Since $\S_e$ and $\pi$ are perfectoid spaces, $\S$ is also a perfectoid space, with underlying topological space
	\[|\S|=\coprod_{i\in \pi}|\S_e|\simeq \varprojlim_K |S_K^{ad}|.\] The condition on rings in Definition \ref{def} can be verified directly by using the relation \[S_K^{ad}=\coprod_{\gamma \in \pi/K\simeq \Delta\setminus \Gamma/K } (S_e/\gamma K\gamma^{-1}\cap \Delta)^{ad}.\]
	The other assertions are clear.
\end{proof}

By the above Propositions \ref{P:connected} and \ref{P:perf-induction}, the following corollary is immediate.
\begin{corollary}\label{C:P}
	Let $(G,X)$ be a Shimura datum. Fix a prime to $p$ level $K^p\subset G(\A_f^p)$. Consider the statement
	
	\emph{$\mathcal{P}(G,X)$: There exists a perfectoid space $S_{K^p}$ over $\C_p$ such that \[S_{K^p}\sim \varprojlim_{K_p}\Sh_{K_pK^p}(G,X)^{ad}.\]}
	Fix a connected component $X^+\subset X$, and consider the statement
	
	\emph{$\mathcal{P}(G^{der},X^+)$: There exists a perfectoid space $S_{K^p}^0$ over $\C_p$ such that \[S_{K^p}^0\sim \varprojlim_{K_p}\Sh_{K_pK^p}^0(G,X)^{ad}.\]}
	
	Then the two statements are equivalent
	\[\mathcal{P}(G,X)\Longleftrightarrow \mathcal{P}(G^{der},X^+).\]
	
\end{corollary}
\begin{remark}\label{R}
In \cite{Mo} 2.10, there is a list (a)-(e) of properties of a statement $\mathcal{P}(G,X)$ for a Shimura datum $(G, X)$. Let $\mathcal{P}(G,X)$ be as in the above corollary. Then the results of \cite{Sch3} imply that it satisfies (a) and (b) of \cite{Mo} 2.10. Our results in this paper imply that $\mathcal{P}(G,X)$ satisfies (c) and (e) of loc. cit.. Finally, it is not hard to verify that the statement  $\mathcal{P}(G,X)$ also satisfies (d) of loc. cit. on products, cf. Lemma \ref{L:product}. Thus $\mathcal{P}(G,X)$ satisfies the list (a)-(e) of \cite{Mo} 2.10.
\end{remark}

Now we assume that $(G, X)$ is of abelian type. By Definition \ref{D:abel} (2), there exists a Shimura datum $(G_1,X_1)$ of Hodge type, together with a central isogeny $G_1^{der}\ra G^{der}$ which induces an isomorphism $(G_1^{ad},X^{ad}_1)\st{\sim}{\ra}(G^{ad},X^{ad})$. These data put us into the situation of the last subsection by setting $(G,X)=(G_2,X_2)$. Fix open compact subgroups $K^p\subset G(\A_f^p)$ and $K_1^p\subset G_1(\A_f^p)$ such that the isogeny induces a map $K_1^p\cap G_1^{der}(\A_f^p)\ra K^p\cap G^{der}(\A_f^p)$. Since $(G_1,X_1)$ is of Hodge type, by Theorem~\ref{thm}, there is a perfectoid Shimua variety $S_{K_1^p}(G_1,X_1)$ such that \[S_{K_1^p}(G_1,X_1)\sim \varprojlim_{K_{1p}}\Sh_{K_{1p}K_1^p}(G_1,X_1)^{ad}.\] Applying Corollary \ref{C:connected}, we get the connected perfectoid Shimura variety $S_{K^p_1}^0(G_1,X_1)$ equipped with an action of $\Al(K^p_1)^0=(\Gamma_0^1)^\wedge\big(\tr{rel}. \tau(G_1^{der})\big)$, where $\Gamma_0^1=\tr{Im}\Big([G^{der}_1(\Q)_+\cap K_1^p]\ra G^{ad}_1(\Q)^+\Big)$, such that \[S_{K^p_1}^0(G_1,X_1)\sim \varprojlim_{K_{1p}}\Sh_{K_{1p}K^p_1}^0(G_1,X_1)^{ad}.\]Moreover, there is an action of $\Al(K^p_1)$ on $S_{K_1^p}(G_1,X_1)$, compatible with the action of $\Al(K^p_1)^0$ on $S_{K^p_1}^0(G_1,X_1)$.
 Consider the group
\[\Delta=\varprojlim_{K_{1p}}\frac{\tr{Im}\Big([G^{der}(\Q)_+\cap \pi(K_{1p}^{der})K^p]\ra G^{ad}(\Q)^+\Big)}{\tr{Im}\Big([G^{der}_1(\Q)_+\cap K_{1p}K_1^p]\ra G^{ad}_1(\Q)^+\Big)}.\]
By Proposition \ref{P:delta}, this is a finite group. For $K_{1p}$ sufficiently small, $\Delta$ acts freely on $\Sh_{K_{1p}K^p_1}^0(G_1,X_1)^{ad}$. Therefore, it also acts freely on $S_{K_1^p}^0(G_1,X_1)$. We want to take the quotient \[S_{K_1^p}^0(G_1,X_1)/\Delta.\] The following proposition says that such a quotient $S_{K^p}^0(G, X)=S_{K_1^p}^0(G_1,X_1)/\Delta$ indeed exists, cf. Corollary \ref{C:Quotient}.
\begin{proposition}\label{P:abelian-connected}
	There exists a perfectoid space $S_{K^p}^0(G, X)$ over $\C_p$, such that \[S_{K^p}^0(G, X)\sim \varprojlim_{K_p}\Sh^0_{K_pK^p}(G,X)^{ad}.\]
 We have a finite \'etale Galois morphism  \[S_{K_1^p}^0(G_1,X_1)\ra S_{K^p}^0(G, X)\]
	with Galois group $\Delta$.
\end{proposition}
\begin{proof}
	We verify that the conditions in Proposition \ref{prop} and Corollary \ref{C:Quotient} hold. By construction, the action of $\Delta$ on $S_{K_1^p}^0(G_1,X_1)$ comes from a system of finite \'etale Galois cover $X_{K_{1p}}\ra Y_{K_{1p}}$ with Galois group $\Delta(K_{1p})$, such that \[\Sh_{K_1^p}^0(G_1, X_1)=\varprojlim_{K_{1p}}X_{K_{1p}},\quad \Delta=\varprojlim_{K_{1p}}\Delta(K_{1p}),\] see the last paragraph in the subsection \ref{S}. By the description there, for each $K_{1p}\subset G_1(\Q_p)$, we have a $\Delta(K_{1p})$-torsor $X_{K_{1p}}\ra Y_{K_{1p}}$ of varieties over $\C_p$. By GAGA, the morphism of the associated adic spaces $X_{K_{1p}}^{ad}\ra Y_{K_{1p}}^{ad}$ is a $\Delta(K_{1p})$-torsor. As $S_{K^p_1}^0(G_1,X_1)\sim \varprojlim_{K_{1p}}X_{K_{1p}}^{ad}$, and $\Delta$ is finite, we see that the conditions in Proposition \ref{prop} and Corollary \ref{C:Quotient} hold true. Therefore,  $S_{K^p}^0(G, X):=S_{K_1^p}^0(G_1,X_1)/\Delta$ is a perfectoid space, \[S_{K^p}^0(G, X)\sim\varprojlim_{K_{1p}}Y_{K_{1p}}^{ad}=\varprojlim_{K_p}\Sh^0_{K_pK^p}(G,X)^{ad},\] and we have a finite \'etale Galois morphism $S_{K_1^p}^0(G_1,X_1)\ra S_{K^p}^0(G, X)$ with Galois group $\Delta$.
\end{proof}

Finally, we can prove the following theorem, which asserts that Shimura varieties of abelian type with infinite level at $p$ are perfectoid. In the next subsection we will discuss the Hodge-Tate period map.
\begin{theorem}\label{S:thm}
	Assume that $(G,X)$ is an abelian type Shimura datum.
There is a perfectoid space $S_{K^p}=S_{K^p}(G,X)$ over $\C_p$ such that
\[S_{K^p}\sim \varprojlim_{K_p}\Sh_{K_pK^p}(G,X)^{ad},\]where $\Sh_{K_pK^p}(G,X)^{ad}$ is the adic space associated to $\Sh_{K_pK^p}$ over $\C_p$.
\end{theorem}
\begin{proof}
This is a direct consequence of our Propositions \ref{P:abelian-connected} and \ref{P:perf-induction}.
\end{proof}
We remark that, by \cite{SW} Proposition 2.4.5, the perfectoid space $S_{K^p}$ in the above theorem does not depend on the choices of the Hodge type Shimura datum $(G_1, X_1)$ and the prime to $p$ level $K_1^p$.

Going through the arguments that we used, combined with the results of \cite{MS} section 3, we get the following corollary.
\begin{corollary}\label{C:shim}
Let the situation be as in Corollary \ref{C:P}. Then to prove the statement $\mathcal{P}(G, X)$, it suffices to prove the case for a Shimura datum $(G, X)$ with $G^{der}$ simply connected.
\end{corollary}

Let the situation be as in Theorem \ref{S:thm}. Recall that we have fixed the prime to $p$ level $K^p$. Now let $K^p$ vary, we get a family of perfectoid Shimura varieties $(S_{K^p})_{K^p}$ over $\C_p$. As usual, we get a prime to $p$ Hecke action on this family of perfectoid spaces. More precisely, if $(K^p)'\subset K^p$ be another open compact subgroup, we get a natural morphism $S_{(K^p)'}\ra S_{K^p}$, which is a finite \'etale morphism of perfectoid spaces. For any $\gamma\in G(\A_f^p)$, we get an isomorphism of perfectoid spaces $S_{\gamma^{-1}K^p\gamma}\ra S_{K^p}$. In particular, for $\gamma\in G(\A_f^p)$, we get a prime to $p$ Hecke correspondence of perfectoid spaces
\[\xymatrix{
	& S_{K^p\cap \gamma^{-1}K^p\gamma}\ar[ld]_{p_1}\ar[rd]^{p_2}&\\
	S_{K^p}& &S_{K^p},
	}\]
where $p_1$ is the natural projection, $p_2$ is the composite of the natural projection $S_{K^p\cap \gamma^{-1}K^p\gamma}\ra S_{ \gamma^{-1}K^p\gamma}$ with the isomorphism $S_{\gamma^{-1}K^p\gamma}\simeq S_{K^p}$.

We would like to give a corollary on completed cohomology. Let $(G,X)$ be a Shimura datum such that the associated Shimura varieties $\Sh_K$ are proper (i.e. $G$ is anisotropic modulo center). Fix a tame level $K^p\subset G(\A_f^p)$ and an integer $n\geq 1$. We consider the following cohomology groups
\[H^i(K^p,\Z/p^n\Z)=\varinjlim_{K_p}H^i_{\tr{\'et}}(\Sh_{K_pK^p,\ov{\Q}},\Z/p^n\Z),\quad \wt{H}^i(K^p)=\varprojlim_m\varinjlim_{K_p}H^i_{\tr{\'et}}(\Sh_{K_pK^p,\ov{\Q}},\Z/p^m\Z).\]
\begin{corollary}\label{C:cohomology}
	Let $(G,X)$ be a Shimura datum of abelian type such that the associated Shimura varieties are proper, e.g. the non PEL type quaternionic Shimura vaieties in Example \ref{E:quaternion}. With the above notations, for $i> \tr{dim}\,\Sh_K$, we have
	\[H^i(K^p,\Z/p^n\Z)=\wt{H}^i(K^p)=0.\]
\end{corollary}
\begin{proof}
Identical to the proof of Corollary 6.2 in \cite{Sch5}; see also Theorem IV. 2.1 in \cite{Sch3} and Theorem 17.2 in \cite{Sch6}. The key point is that under our assumption, by Theorem \ref{S:thm} $S_{K^p}\sim \varprojlim_{K_p}\Sh_{K_pK^p}^{ad}$, which gives
\[\varinjlim_{K_p}H^i_{\tr{\'et}}(\Sh_{K_pK^p}^{ad},\mathcal{O}^+_{\Sh_{K_pK^p}^{ad}}/p^n)=H^i_{\tr{\'et}}(S_{K^p},\mathcal{O}^+_{S_{K^p}}/p^n),\]
cf. \cite{Sch1} Corollary 7.18. Then, the \'etale cohomology on the right hand is almost equal to the coherent cohomology $H^i_{\tr{an}}(S_{K^p},\mathcal{O}^+_{S_{K^p}}/p^n)$, and the latter admits the standard bound $\tr{dim}\,\Sh_K$ for the cohomological dimension. We refer to the proofs of Corollary 6.2 in \cite{Sch5} and of Theorem IV. 2.1 in \cite{Sch3} for more details.
\end{proof}
We can also give a generalization of Theorem IV. 3.1 of \cite{Sch3} to the above situation, which says roughly that all Hecke eigenvalues appearing in $\wt{H}^i(K^p)$ come via $p$-adic interpolation from Hecke eigenvalues in the space of classical automorphic forms. However, we will not do this task at this moment. As remarked in the introduction, we will prefer to first extend our Theorem \ref{S:thm} to some suitable compactifications (e.g. the minimal compactification). Then we will generalize our Corollary \ref{C:cohomology} and the Theorem IV. 3.1 of \cite{Sch3} in that setting.

\subsection{The Hodge-Tate period map}
As in the last subsection, let $(G,X)$ be a Shimura datum of abelian type. We take a Shimura datum $(G_1, X_1)$ of Hodge type, together with a central isogeny $G_1^{der}\ra G^{der}$ which induces an isomorphism $(G_1^{ad},X^{ad}_1)\st{\sim}{\ra}(G^{ad},X^{ad})$. Let $K^p, K_1^p$ be the prime to $p$ levels as in the last subsection. Recall that for a Shimura datum $(G,X)$, we can associate to it a conjugacy class $\{\mu\}$ of cocharacters $\mu: \G_m\ra G_{\C_p}$, which is in fact defined over the local reflex field inside $\C_p$. Fix a choice $\mu\in\{\mu\}$, and let $P_\mu$ be the parabolic subgroup of $G_{\C_p}$ which stabilizes the filtration opposite to the usual Hodge filtration attached to $\mu$. One can also define it over $\C$ first directly as \[ P_\mu=\{g\in G_{\C}|\,\lim_{t\ra 0}ad(\mu(t))g \,\mathrm{exists} \},\] where $ad(h)g=hgh^{-1}$ for any $h, g\in G_\C$. Then this parabolic is defined over the local reflex field $E$, and we can pull back it to a parabolic subgroup of $G$ over $\C_p$.  Set \[\Fl_G=(G_{\C_p}/P_\mu)^{ad},\] the adic space associated to the flag variety $G_{\C_p}/P_\mu$ over $\C_p$. For our Shimura datum of abelian type $(G, X)$ and the datum $(G_1, X_1)$ of Hodge type as above, the isomorphism $G^{ad}\simeq G_1^{ad}$ induces an isomorphism of the $p$-adic flag varieties $\Fl_G\simeq \Fl_{G_1}$. By \cite{CS}, there exists a $G_1(\Q_p)$-equivariant Hodge-Tate period map
\[S_{K_1^p}(G_1, X_1)\ra \Fl_{G_1}\]of adic spaces over $\C_p$. In the Siegel case $(G_1, X_1)=(\GSp_{2g}, S^\pm)$, this map sends a $\C_p$-point of $S_{K_1^p}(G_1, X_1)$ to the Hodge-Tate filtration of its associated abelian variety $A$ over $\C_p$ (cf. \cite{Sch3} III.3)
\[0\ra \Lie A(1)\ra T_p(A)\otimes_{\Z_p}\C_p\ra (\Lie A^\vee)^\vee\ra0.\]Here $A^\vee$ is the dual abelian variety of $A$. In the general Hodge type datum case, one should also consider the additional Hodge tensors defining $G$.

Recall the group $\Al(K_1^p)$ defined in the subsection 3.2. By construction, we have a natural projection $\Al(K_1^p)\ra G_1(\Q_p)$. Let $\Al(K_1^p)$ act on $\Fl_{G_1}$ via this projection. Then the Hodge-Tate period map $S_{K_1^p}(G_1, X_1)\ra \Fl_{G_1}$ is $\Al(K_1^p)$-equivariant. We note that the group $\Delta$ defined as in the last subsection acts trivially on $\Fl_{G_1}$. Let $M_\mu\subset P_\mu$ be the centralizer of $\mu$, which is a Levi subgroup of $P_\mu$. Let $\tr{Rep}\,M_\mu$ be the category of algebraic representations of $M_\mu$ over $\C_p$. The Shimura variety with level $K_pK^p$ (over $\C_p$) associated to the datum $(G, X)$ will be denoted simply by $\Sh_{K_pK^p}$, and the perfectoid Shimura varieties will be denoted by $S_{K^p}$ and $S^0_{K^p}$. For expositional simplicity, we assume that the maximal $\Q$-anisotropic $\R$-split subtorus $Z_s$ of the center $Z$ of $G$ is trivial. (Otherwise, we shall assume instead that all representations we consider below have trivial restrictions to $Z_s$, cf. \cite{Mi1} III.8.)
\begin{theorem}\label{HT}
\begin{enumerate}
\item There is a $G(\Q_p)$-equivariant map of adic spaces
\[\pi_{\emph{HT}}: S_{K^p}\lra \Fl_G,\]compatible with the construction in \cite{CS} in the case $(G, X)$ is of Hodge type. The map $\pi_{\emph{HT}}$ is invariant for the prime to $p$ Hecke action on $S_{K^p}$, and does not depend on any other choices.
\item The pullbacks of automorphic vector bundles over finite level Shimura varieties to $S_{K^p}$ can be understood by the map $\pi_{\emph{HT}}$. More precisely, there is an isomorphism of tensor functors \[\mathrm{Rep}\,M_\mu\ra \{G(\A_f)-\mathrm{equiv.\, vector\, bundles\, on\,} S_{K^p}\}\] given by
    \[\xymatrix{
    \mathrm{Rep}\,M_\mu\ar[r]\ar[d]&\{G(\Q_p)-\mathrm{equiv.\, vector\, bundles\, on\,} \Fl_G\}\ar[d]^{\pi_{\emph{HT}}^\ast}\\
\{\mathrm{Auto.\, vector\, bundles\, on\,} \Sh_{K_pK^p}\}\ar[r] &\{G(\A_f)-\mathrm{equiv.\, vector\, bundles\, on\,} S_{K^p}\}.
 }\]Here, the upper horizontal (resp. left vertical)  arrow is the usual construction of vector bundles (resp. automorphic vector bundles) on $\Fl_G$ (resp. $\Sh_{K_pK^p}$), the lower horizontal arrow is the composition of the GAGA functor (which identifies vector bundles on $\Sh_{K_pK^p}$ and on $\Sh_{K_pK^p}^{ad}$) and the pullback functor $f^\ast_{K_p}$ associated to the projection $f_{K_p}: S_{K^p}\ra \Sh_{K_pK^p}^{ad}$, and finally, the right vertical arrow is the pullback functor $\pi_{\emph{HT}}^\ast$ associated to the Hodge-Tate period map $\pi_{\emph{HT}}$.
\end{enumerate}
\end{theorem}
\begin{proof}
Both are deduced from Theorem 2.1.3 in \cite{CS}. 

For (1), we note that, when restricting the map $S_{K_1^p}(G_1, X_1)\ra \Fl_{G_1}$ to the subspace $S_{K_1^p}^0(G_1, X_1)$, it is $\Delta$-equivariant, as the group action of $\Delta$ on $S_{K_1^p}^0(G_1, X_1)$ is induced by the Hecke action. Since $\Delta$ acts trivially on $\Fl_{G_1}$, the map $S_{K_1^p}^0(G_1, X_1)\ra \Fl_{G_1}$  factors through $S_{K^p}^0$. Applying the $\Al(K^p)$-action, we get an extension of $S_{K^p}^0\ra \Fl_{G}=\Fl_{G_1}$ to a map $\pi_{\tr{HT}}: S_{K^p}\lra \Fl_G$, which is $G(\Q_p)$-equivariant. By construction, $\pi_{\tr{HT}}$ is invariant for the prime to $p$ Hecke action on $S_{K^p}$, and it does not depend on the choices of Hodge type Shimura datum $(G_1, X_1)$ and the prime to $p$ level $K_1^p$, cf. Remark \ref{R:abelian}.

For (2), we may reduce the problem to the corresponding one for automorphic vector bundles over connected Shimura varieties (cf. \cite{Mi1}). Then it follows from the result of \cite{CS} for the case of $(G_1, X_1)$.
\end{proof}

At this point, we remark that in \cite{Han} Hansen has recently constructed the Hodge-Tate period map for general Shimura varieties, by working on the diamonds associated to Shimura varieties, and by applying the theorem of Liu-Zhu that the tautological $\Q_p$-local systems on Shimura varieties are de Rham (cf. \cite{LZ} Theorem 1.2). When the Shimura datum $(G, X)$ is of abelian type, our Hodge-Tate period map above $\pi_{\tr{HT}}$ coincides with Hansen's after passing to diamonds, cf. subsection 1.2 of \cite{Han}.

To conclude this subsection, we would like to discuss an example of the Hodge-Tate period map, which is related to the one studied in section 6 of \cite{Sch7}. Let $F$ and $D$ be as in Example \ref{E:quaternion}, and $\Sh_{K_pK^p}^{ad}$ be the adic quaterionic Shimura varieties which are not of PEL (Hodge) type. Fix a tame level $K^p$. For simplicity, assume that $p$ is inert in $F$. Suppose that $D$ is ramified at $p$ and split at only one archimedean place $\infty_F$. Then $\tr{dim}\Sh_K=1$, and the associated $p$-adic flag variety is just the $p$-adic projective line $\mathbb{P}^1_{\C_p}$. Let $S_{K^p}$ be the associated perfectoid Shimura curve of abelian type, and 
\[\pi_{\tr{HT}}: S_{K^p}\ra \mathbb{P}^1_{\C_p}\] be the Hodge-Tate period map constructed by Theorem \ref{HT}. On the other hand, by our assumption that $D$ is ramified at $p$, the theorem of Cerednik (cf. \cite{BZ} Corollary 3.4) on $p$-adic uniformization of Shimura curves implies that
\[S_{K^p}\simeq G(F)\setminus \mathcal{M}_{\tr{Dr}, \infty, \C_p}\times G(\A_{F,f}^p)/K^p,\]where $G=(D')^\times$ is the multiplicative group of another quaternion algebra $D'$ which is locally isomorphic to $D$ at all places outside $p$ and $\infty_F$, and such that $D'_p$ is split while $D'_{\infty_F}$ is ramified. The space $\mathcal{M}_{\tr{Dr}, \infty, \C_p}$ is the perfectoid Drinfeld space associated to $D_p$ over $\C_p$ (cf. \cite{SW} 6.3). By \cite{SW} Proposition 7.1.1, there is a Hodge-Tate period map
\[\pi_{\tr{HT}}^{\mathcal{M}}: \mathcal{M}_{\tr{Dr}, \infty, \C_p}\ra \mathbb{P}^1_{\C_p},\]which induces a map
\[
\pi_{\tr{HT}}^{\Sh}: S_{K^p}\ra \mathbb{P}^1_{\C_p}\] by definition and the construction of $p$-adic uniformization. The following claim confirms the concern in Remark 6.4 of \cite{Sch7}.
\begin{claim}
We have the identity of morphisms from $S_{K^p}$ to $\mathbb{P}^1_{\C_p}$
\[\pi_{\emph{HT}}=\pi_{\emph{HT}}^{\Sh}.\]
\end{claim}
\begin{proof}
Take a totally imaginary CM extension $E$ of $F$, and consider the simple unitary group $G_1$ over $\Q$ associated to $D\otimes_FE$ (cf. \cite{BZ} p. 47, where it is denoted by $G^\bullet$). $G_1$ sits in the following exact sequence
\[1\ra Z\ra D^\times\times Z_E\ra G_1\ra 1,\]where $Z=Res_{F|\Q}\G_{m,F}, Z_E=Res_{E|\Q}\G_{m,E}$, the map $Z\ra D^\times\times Z_E$ is given by $f\mapsto (f,f^{-1})$, and $D^\times\times Z_E\ra G_1$ is given by $(d, k)\mapsto d\otimes k$. Write $(G_2, X_2)$ as the Shimura datum (of abelian type) associated to $D$. After choosing a morphism $h_E: \mathbb{S}\ra Z_{E, \R}$, we get the conjugacy class $X_1$ of morphisms $\mathbb{S}\ra G_{1\R}$ by the using the above sequence. Then $(G_1, X_1)$ forms a PEL type Shimrua datum. Moreover, we have $G_1^{der}=G_2^{der}$. We choose a prime to $p$ level $K_1^p\subset G_1(\A_f^p)$ constructed from $K_2^p=K^p$ by (3.10) of \cite{BZ}.  By the proof of Theorem \ref{HT}, we can use the Hodge-Tate period map for $S_{K_1^p}(G_1, X_1)$ to define the Hodge-Tate period map
\[\pi_{\tr{HT}}: S_{K_2^p}(G_2, X_2)=S_{K^p}\ra \mathbb{P}^1_{\C_p}.\] In fact, since $G_1^{der}=G_2^{der}$, we have the equality of geometric connected components \[\Sh_{K_{1p}K_1^p}^0(G_1, X_1)_{\C_p}=\Sh_{K_{2p}K_2^p}^0(G_2, X_2)_{\C_p}\] for compatible $K_{1p}K_1^p$ and $K_{2p}K_2^p$. That is, the finite group $\Delta$ in the last subsection is trivial, and \[S_{K_1^p}^0(G_1,X_1)=S_{K_2^p}^0(G_2, X_2).\] For the unitary Shimura curves $\Sh_{K_{1p}K_1^p}^0(G_1, X_1)_{\C_p}$, we have also the global $p$-adic uniformization by $\mathcal{M}_{\tr{Dr}, \infty, \C_p}$ (cf. \cite{BZ} 1.51 and (3.8)). Then, we can reduce the claim to the corresponding identity for the perfectoid unitary Shimura curve $S_{K_1^p}(G_1, X_1)$. For the latter case, one sees easily that the identity holds, by the definition of the Hodge-Tate period map in the PEL case and the construction of $p$-adic uniformization.
\end{proof}

\section{Application to moduli spaces of polarized K3 surfaces}
In this section, we investigate some special (and very interesting) examples of Shimura varieties of abelian type, namely, the Shimura varieties associated to the orthogonal group $\SO(V)$, with $V$ a $\Q$-vector space equipped with a non degenerate quadratic form $Q$ such that $V_\R$ has signature $(2,n)$ for some integer $n\geq 1$. These Shimura varieties appear in Kudla's program on special cycles and their generating series, see \cite{Ku} for example. The case that $n=19$ is closely related to moduli spaces of polarized K3 surfaces. In particular, we can apply results of the last section to prove that the moduli spaces of polarized K3 surfaces with infinite level at $p$ are perfectoid.

\subsection{An example: Shimura varieties of orthogonal type}

Let $R$ be a commutative ring in which 2 is invertible, and $(L,Q)$ be a non-degenerate quadratic space over $R$. We have the associated Clifford algebra $C:=C(L)$ over $R$, equipped with an embedding $L\hookrightarrow C$ and a natural decomposition $C=C^+\oplus C^-$, so that $C^+$ is a sub-algebra of $C$. We define a reductive group scheme $G_1:=\GSpin(L,Q)$ over $R$ as follows. For any $R$-algebra $S$, 
\[G_1(S)=\{x\in (C_S^+)^\times|\, x(L_S)x^{-1}=L_S\}.\] Let $G=\SO(L,Q)$ be the special orthogonal group. Then there is an exact sequence of group schemes over $R$:
\[0\ra \G_m\ra G_1\ra G\ra 0,\]where the natural morphism of group schemes $G_1\ra G$ is given by $g\mapsto (v\mapsto gvg^{-1})$. On the other hand, there is a canonical character, the spinor norm, $\nu: G_1\ra \G_m$. Let $G_1'$ be the kernel of $\nu$, which is usually denoted by $\Spin(L,Q)$. Then $G_1'=G_1^{der}$ is the derived subgroup of $G_1$, which is simply connected, and we have the following exact sequence of groups over $R$
\[0\ra G_1'\ra G_1\ra \G_m\ra 0.\] Moreover, the morphism $G_1\ra G$ induces the following exact sequence of groups over $R$
\[0\ra \mu_2\ra G_1'\ra G\ra 0.\]

Consider the case $R=\Q$. We denote $L=V$ as a vector space over $\Q$. Assume that the quadratic space $(V,Q)$ has signature $(2,n)$ over $\R$ for some $n\geq 1$.  We have the reductive groups $G_1$ and $G$ over $\Q$. Let $X$ be the space of oriented negative definite 2-planes in $V_\R$. The points of $X$ correspond to certain Hodge structures of weight 0 on the vector space $V$, polarized by $Q$, cf. \cite{MP0} 3.1. The pairs $(G_1,X)$ and $(G,X)$ form Shimura data, both of which have reflex field $\Q$ since $n\geq1$. 

For any $\delta\in (C^+)^\times$ such that $\delta^\ast=-\delta$, the form \[(x, y)_\delta\mapsto \tr{tr}(x\delta y^\ast)\]on $C^+$ is symplectic, non-degenerate. Here $\ast$ is the canonical involution on $C$. Thus this form $(\cdot,\cdot)_\delta$ induces an closed embedding of group schemes $G_1\hookrightarrow \GSp(C^+, (\cdot,\cdot)_\delta)$, and an embedding of Shimura data $(G_1, X)\hookrightarrow (\GSp(C^+, (\cdot,\cdot)_\delta), S^{\pm})$. Accordingly, the Shimura datum $(G_1,X)$ is of Hodge type, and $(G,X)$ is of abelian type.

Let $K_1\subset G_1(\A_f)$ be an open compact subgroup. We will assume that $K_1$ is of the form $K_{1p}K^p_1$ with $K_{1p}\subset G_1(\Q_p)$ and $K^p_1\subset G_1(\A_f^p)$. As before, we assume that $K^p_1$ is chosen to be sufficiently small and fixed. We will only let the level $K_{1p}$ at $p$ vary. Let $K\subset G(\A_f)$ be the image of $K_1$ under the above map $G_1\ra G$. It has the form as $K_pK^p$ with $K_p, K^p$ the images of $K_{1,p}, K^p_1$ respectively. Now we get the associated Shimura varieties $\Sh_{K_1}(G_1,X), \Sh_{K}(G,X)$ over $\Q$, and a morphism \[\Sh_{K_1}(G_1,X)\ra \Sh_{K}(G,X).\] We have identifications of complex analytic varieties:
\[\Sh_{K_1}(G_1,X)_\C^{an}=G_1(\Q)_+\setminus X^+\times G_1(\A_f)/K_1,\]
\[\Sh_{K}(G,X)_\C^{an}=G(\Q)_+\setminus X^+\times G(\A_f)/K.\]Since $H^1(k, \G_m)=0$ for any field $k$ of characteristic zero (Hilbert's Theorem 90), we find that the map $\Sh_{K_1}(G_1,X)\ra \Sh_{K}(G,X)$ is a finite \'etale Galois cover with Galois group
\[\Delta(K_{1p})=\Q^\times\setminus\alpha^{-1}(K)/K_1,\]where $\alpha: G_1\ra G $ is the morphism $g\mapsto (v\mapsto gvg^{-1})$ as before, cf. \cite{Mi4} Lemma 4.13. If $K_{1p}'\subset K_{1p}$ is another open compact subgroup, we have the following commutative diagram of morphisms
\[\xymatrix{\Sh_{K_{1p}'K^p_1}(G_1,X)\ar[r]\ar[d]& \Sh_{K_{1p}K^p_1}(G_1,X)\ar[d]\\
	\Sh_{K_p'K^p}(G,X)\ar[r]&\Sh_{K_pK^p}(G,X).
	} \]Moreover, we have a surjective morphism of finite groups $\Delta(K_{1p}')\ra \Delta(K_{1p})$.
Since the weight of $(G,X)$ is defined over $\Q$, by Theorem 3.31 of \cite{Mi2}, the associated Shimura varieties are moduli spaces of some abelian motives. For the Shimura varieties of Hodge type associated to $(G_1, X)$, we have the following list when $n$ is small.
\begin{examples}
	\begin{enumerate}
		\item	In the case that $n$ is small, $(G_1,X)$ is of PEL type: the varieties $\Sh_{K_1}(G_1,X)$ are
		\begin{itemize}
			\item $n=1$, modular curves and Shimura curves,
			\item $n=2$, Hilbert modular surfaces and their quaternionic versions,
			\item $n=3$, Siegel threefolds and their quaternionic versions. 
		\end{itemize}
		\item In the next subsection, we will investigate the case $n=19$ (for $(G, X)$).
	\end{enumerate}
\end{examples}

We keep the notations as above. Consider the varieties $\Sh_{K_1}(G_1,X), \Sh_{K}(G,X)$ over $\C_p$. Recall the spinor norm $\nu: G_1\ra \G_m$. We have the following description of the sets of connected components:
\[\begin{split} \pi_0(\Sh_{K_1}(G_1,X))&=  G_1(\Q)_+\setminus G_1(\A_f)/K_1\\&=\Q^\times \R^\times_{>0}\setminus \A^\times/\nu(K_1)\\&= \Q_{>0}\setminus \A^\times_f/\nu(K_1), \end{split}\]
\[ \pi_0(\Sh_{K}(G,X))= G(\Q)_+\setminus G(\A_f)/K.\]In this special case, the morphism $\alpha$ induces a surjection $G_1(\A_f)\twoheadrightarrow G(\A_f)$ and thus a surjection $\pi_0(\Sh_{K_1}(G_1,X))\twoheadrightarrow \pi_0(\Sh_{K}(G,X))$.
One sees that $\pi_0(\Sh_{K_1}(G_1,X))$ is a $\Delta(K_{1p})$-torsor over $\pi_0(\Sh_{K}(G,X))$.
The map on the connected Shimura varieties
\[\Sh^0_{K_1}(G_1,X)\ra \Sh_{K}^0(G,X) \]
is an isomorphism, since \[\alpha\Big(K_1\cap G_1(\Q)_+\Big)=K\cap G(\Q)_+.\] Using the notation of of the last section, we have
$\Delta=\{id\}$.
By GAGA, the associated map between the corresponding adic spaces
\[\Sh^0_{K_1}(G_1,X)^{ad}\ra \Sh_{K}^0(G,X)^{ad}\] is an isomorphism. Let $K_{1p}$ vary. We get the connected perfectoid Shimura variety $S^0_{K_1^p}(G_1,X)$ such that
\[S^0_{K_1^p}(G_1,X)\sim \varprojlim_{K_{1p}}\Sh^0_{K_{1p}K_1^p}(G_1,X)^{ad}.\]
Moreover,  
we get the perfectoid space $S^0_{K^p}(G,X)$ which is isomorphic to $S^0_{K_1^p}(G_1,X)$ such that
\[S^0_{K^p}(G,X)\sim  \varprojlim_{K_{p}}\Sh^0_{K_{p}K^p}(G,X)^{ad}.\] 
By Proposition \ref{P:perf-induction}, we find the perfectoid Shimura variety
$S_{K^p}(G,X)$ such that \[S_{K^p}(G,X)\sim  \varprojlim_{K_{p}}\Sh_{K_{p}K^p}(G,X)^{ad}.\]

\subsection{Moduli spaces of polarized K3 surfaces and the period map}
In this subsection, we specialize our Shimura varieties in the last subsection further to the case $n=19$, i.e. the case of orthogonal Shimura varieties associated to $\SO(2,19)$. We will discuss the relation with moduli spaces of polarized K3 surfaces, cf. \cite{MP} sections 2 and 4, \cite{Ri2} section 6.

Let $U$ be the hyperbolic lattice over $\Z$ of rank 2, and $E_8$ be the positive quadratic lattice associated to the Dynkin diagram of type $E_8$. Set $N=U^{\oplus3}\oplus E_8^{\oplus2}$, which is a self-dual lattice. Let $d\geq 1$ be an integer. Choose a basis $e, f$ for the first copy of $U$ in $N$ and set
\[L_d=\lan e-df\ran^\bot\subset N.\]This is a quadratic lattice over $\Z$ of discriminant $2d$ and rank 21 (in \cite{Ri2} it is denoted by $L_{2d}$ ). Let $V_d=L_d\otimes\Q$ and $L_d^\vee\subset V_d$ be the dual lattice. Set \[G=\SO(V_d),\] which is isomorphic to the special orthogonal group over $\Q$ of signature $(2, 19)$. Let $K\subset G(\A_f)$ be an open compact subgroup which stabilizes $L_{d,\wh{\Z}}$ and acts trivially on $L_{d,\wh{\Z}}^\vee/L_{d,\wh{\Z}}$. Such compact opens are called \emph{admissible}. We fix an open compact subgroup $K^p$ from now on. We only consider open compact subgroups $K_p\subset G(\Q_p)$ which is contained in the discriminant kernel of $L_{d,\Z_p}$ (i.e. the maximal subgroup of $G(\Q_p)$ which stabilizes $L_{d,\Z_p}$ and acts trivially on $L_{d,\Z_p}^\vee/L_{d,\Z_p}$) with finite index. In particular, $K_pK^p$ is admissible, cf. \cite{Ri2} 5.3. For the reductive group $G$, we have the associated Shimura varieties $\Sh_{K_pK^p}$, which are defined over $\Q$. 

Let $k$ be a field. Recall that a K3 surface $X$ over $k$ is a projective smooth surface over $k$ such that $\Omega^2_{X/k}\simeq \mathcal{O}_X$ and $H^1(X,\mathcal{O}_X)=0$. Recall also that (cf. \cite{Ri2} Definition 1.1.5) a K3 surface $X$ over a scheme $S$ is a proper smooth morphism $f: X\ra S$ of schemes whose geometric fibers are K3 surfaces. We can extend this definition to algebraic spaces: a proper smooth algebraic space $f: X\ra S$ over a scheme $S$ is called a K3 surface if there is an \'etale cover $S'\ra S$ such that the pullback of $f$ to $S'$ is a K3 surface in the above sense.
A primitive polarization (resp. quasi-polarization) on a K3 surface $f: X\ra S$ as above is a global section $\lambda\in \mathrm{Pic}_{X/S}(S)$ such that for every geometric point $\ov{s}$ of $S$ the section $\lambda_{\ov{s}}\in \mathrm{Pic}_{X_{\ov{s}}/k(\ov{s})}(k(\ov{s}))$ is a primitive polarization (quasi-polarization) of $X_{\ov{s}}$. 

Let $\M_{2d}$ (resp. $\M^\ast_{2d}$) be the moduli spaces of K3 surfaces $f: X\ra S$ together with a primitive polarization $\xi$ (resp. quasi-polarization) of degree $2d$ over $\Q$ (in \cite{MP} section 2, these spaces are denoted by $\M_{2d}^\circ$ and $\M_{2d}$ respectively). These are Deligne-Mumford stacks of finite type over $\Q$. The natural map $\M_{2d}\ra \M_{2d}^\ast$ is an open immersion. Moreover, $\M_{2d}$ is separated and smooth of dimension 19 over $\Q$, cf. \cite{Ri2} Theorem 4.3.3, Proposition 4.3.11 and \cite{MP} Proposition 2.2.

Let $(f: \mathcal{X}\ra \M_{2d}, \xi)$ be the universal object over $\M_{2d}$. For any prime $\ell$, we consider the second relative \'etale cohomology $H_\ell^2$ of $\mathcal{X}$ over $\M_{2d}$. This is a lisse $\Z_\ell$-sheaf of rank 22 equipped with a perfect symmetric Poincar\'e pairing $\lan,\ran: H_\ell^2\times H_\ell^2\ra \Z_\ell(-2)$. The $\ell$-adic Chern class $\ch_\ell(\xi)$ of $\xi$ is a global section of the Tate twist $H^2_\ell(1)$ that satisfies $\lan \ch_\ell(\xi),\ch_\ell(\xi)\ran=2d$.  The product \[H^2_{\wh{\Z}}=\prod_{\ell}H_\ell^2\] is a lisse $\wh{\Z}$-sheaf, and the Chern classes of $\xi$ can be put together to get the Chern class $\ch_{\wh{\Z}}(\xi)$ in $H^2_{\wh{\Z}}(1)$. Recall that we have the quadratic lattice $N$ of rank 22 over $\Z$.
\begin{definition}\label{D:level}
Consider the \'etale sheaf over $\M_{2d}$ whose sections over any scheme $T\ra \M_{2d}$ are given by
\[ I(T)=\{ \eta: N\otimes\wh{\Z}\st{\sim}{\ra} H^2_{\wh{\Z},T}(1)\,\tr{isometries},\,\tr{with}\,\eta(e-df)=\ch_{\wh{\Z}}(\xi) \}.\] Let $K\subset G(\A_f)$ be an admissible open compact subgroup. Then $I$ admits a natural action by the constant sheaf of groups $K$. A section $\ov{\eta}\in H^0(T,I/K)$ is called a $K$-level structure over $T$ (in \cite{Ri2} 5.3 it is called a full $K$-level structure). 
\end{definition}
As before, we assume that $K$ is of the form $K_pK^p$.
Let $\M_{2d,K_pK^p}$ (resp. $\M_{2d,K_pK^p}^\ast$) be the relative moduli problem over $\M_{2d}$ (resp. $\M_{2d}^\ast$) which parametrizes $K_pK^p$-level structures. For $K^p$ small enough, these are smooth algebraic spaces. Moreover, the maps 
\[\M_{2d, K_pK^p}\ra \M_{2d},\quad \M_{2d,K_pK^p}^\ast\ra \M_{2d}^\ast\] are finite \'etale. For another admissible $K'=K_p'K^{p'}\subset K$, we have natural finite \'etale projections \[\M_{2d,K'}\ra \M_{2d,K},\quad \M^{\ast}_{2d,K'}\ra \M^\ast_{2d,K}\] as algebraic spaces over $\M_{2d}, \M_{2d}^\ast$ respectively.  When $K'$ is a normal subgroup of $K$, these projections are Galois with Galois group $K/K'$.

For any prime $\ell$, we have the primitive cohomology sheaf \[P_\ell=\lan \ch_\ell(\xi)\ran^\bot \subset H_\ell^2.\]Let $H_B^2$ and $H_{dR}^2$ be the second relative Betti and de Rham cohomology respectively of the universal K3 surface $\mathcal{X}\ra \M_{2d,K_pK^p,\C}^\ast$. We have also the primitive cohomology sheaves
\[P_B=\lan \ch_B(\xi)\ran^\bot \subset H_B^2,\quad P_{dR}=\lan \ch_{dR}(\xi)\ran^\bot\subset H_{dR}^2.\]
Consider $\wt{\M}_{2d,K_pK^p}^\ast\ra \M_{2d,K_pK^p}^\ast$,  the two-fold finite \'etale cover parameterizing isometric trivializations $\det(L_d)\otimes\Z_2\st{\sim}{\ra} \det(P_2)$  of the determinant of the primitive 2-adic cohomology of the universal quasi-polarized K3 surface. We can identify $\wt{\M}_{2d,K_pK^p}^\ast$ with the the space of isometric trivializations $\det(L_d)\st{\sim}{\ra} \det(P_B)$
of the determinant of the primitive Betti cohomology. There is a Hodge-de Rham filtration $F^\bullet P_{dR}$ on $P_{dR}$,  and we have a natural isometric trivialization $\eta: \tr{disc}(L_d)\st{\sim}{\ra} \tr{disc}(P_B)$ and the the tautological trivialization $\beta: \det(L_d)\st{\sim}{\ra}\det(P_B)$. The tuple $(P_B,F^\bullet P_{dR},\eta,\beta)$ gives rise to a natural period map
\[\wt{\M}_{2d,K_pK^p,\C}^\ast\ra \Sh_{K_pK^p,\C},\] cf. \cite{MP} Propositions 4.2 and 3.3. There is a section map $\M_{2d,K_pK^p,\C}\subset \M_{2d,K_pK^p,\C}^\ast\ra \wt{\M}_{2d,K_pK^p,\C}^\ast$, whose composition with the above period map gives us the period map \[\M_{2d,K_pK^p,\C}\lra \Sh_{K_pK^p,\C}.\]
\begin{theorem}\label{Torelli}
The period map
\[\M_{2d,K_pK^p,\C}\lra \Sh_{K_pK^p,\C}\]is defined over $\Q$. Moreover, it is an open immersion.
\end{theorem}
\begin{proof}
The first assertion follows from \cite{Ri1} Theorem 3.9.1, and the second follows from loc. cit. Proposition 2.4.6, which is essentially the global Torelli theorem for K3 surfaces: see loc. cit. for more references therein. 
\end{proof}

In the case $K_p=K_{L_d,p}$,  see also \cite{MP} Corollary 4.4 for the above first assertion, where $K_{L_d}\subset G(\wh{\Z})$ is the largest subgroup which acts trivially on the discriminant $L^\vee_d/L_d$, i.e. the largest admissible open compact subgroup. For the second assertion, in the case $K_p=K_{L_d,p}$, see \cite{MP} Corollary 4.15. For the general case of open compact subgroups $K_p\subset K_{L_d,p}$, we note that the following diagram is cartesian:
\[\xymatrix{ \M_{2d, K_pK^p,\C}\ar[r]\ar[d]& \Sh_{K_pK^p,\C}\ar[d]\\
	\M_{2d, K_{L_d,p}K^p,\C}\ar[r]& \Sh_{K_{L_d,p}K^p,\C}.}\]
As a corollary, we see that for $K^p$ small enough, $\M_{2d, K_pK^p}$ is a scheme for any open compact $K_p\subset K_{L_d,p}$. 

\subsection{The perfectoid moduli spaces of polarized K3 surfaces}
Let the notations be as in the last subsection. Fix an embedding $\Q\subset \C_p$. We consider spaces and varieties over $\C_p$. Let $\M_{2d,K_pK^p}^{ad}$ (resp. $\Sh_{K_pK^p}^{ad}$) be the adic space associated to $\M_{2d,K_pK^p}\times\C_p$ (resp. $\Sh_{K_pK^p}\times\C_p$) for $K_p$ and $K^p$ as above. By Theorem \ref{Torelli}, we get an open subsystem $\M_{2d,K_pK^p}^{ad}\subset \Sh_{K_pK^p}^{ad}$. Applying Lemma \ref{L:sub} and Theorem~\ref{S:thm}, we get the following corollary.
\begin{corollary}
There is a perfectoid space $\M_{2d, K^p}$ over $\C_p$ such that
\[\M_{2d,K^p}\sim \varprojlim_{K_p}\M_{2d, K_pK^p}^{ad}.\]
\end{corollary}

By construction, we have an open immersion of perfectoid spaces over $\C_p$
\[\M_{2d,K^p}\subset S_{K^p}, \]where $S_{K^p}$ is the perfectoid Shimura variety with prime to $p$ level $K^p$ such that $S_{K^p}\sim \varprojlim_{K_p} \Sh_{K_pK^p}^{ad}$. In particular, the restriction on $\M_{2d,K^p}$ of the Hodge-Tate period map $\pi_{\tr{HT}}$ for $S_{K^p}$ gives rise to a Hodge-Tate period map
\[\pi_{\tr{HT}}: \M_{2d,K^p}\ra \Fl_G,\]which can be understood by the Kuga-Satake construction for K3 surfaces (cf. \cite{Ri1} 5.2), and the Hodge-Tate period map for the perfectoid Shimura variety $S_{K_1^p}(G_1,X)$ (cf. \cite{CS}) for the GSpin Shimura datum $(G_1, X)$. Here as before, $K_1^p\subset G_1(\A_f^p)$ is a suitable prime to $p$ level structure, compatible with $K^p$.

In the proof of the corollary above, the key ingredients that we used are perfectoid Shimura varieties of abelian type and the global Torelli theorem for K3 surfaces. With our Theorem \ref{S:thm} at hand, once we have a suitable global Torelli theorem, we can prove that some other moduli spaces of polarized higher dimensional Calabi-Yau varieties with infinite level at $p$ are perfectoid. For example, we can treat the case of moduli spaces of cubic fourfolds exactly as here, see 5.13 and 5.14 of \cite{MP}. In particular, we get some new interesting examples of perfectoid spaces.

\end{document}